\RequirePackage{etex}
\RequirePackage{ifpdf}
\ifpdf
\documentclass[12pt]{amsart}
\else
\documentclass[a4paper,dvipdfmx,12pt]{amsart}
\fi

\usepackage[top=30truemm,bottom=30truemm,left=30truemm,right=30truemm]{geometry}

\usepackage{iftex}
\ifLuaTeX
\RequirePackage{luatex85}
\fi

\usepackage[foot]{amsaddr}
\usepackage[nocompress,noadjust]{cite}

\usepackage{tikz}
\usepackage{tikz-cd}
\usetikzlibrary{cd}
\usetikzlibrary{positioning, arrows.meta, matrix}
\usetikzlibrary{decorations.pathreplacing}
\tikzcdset{
  arrow style=tikz,
  arrows=-stealth
}

\usepackage{lipsum}
\usepackage{adjustbox}
\usepackage{fancybox}
\usepackage{listings}
\usepackage{docmute}
\usepackage{amssymb, enumerate, calligra}
\usepackage{amsthm}
\usepackage{amsmath}
\usepackage{amsxtra}
\usepackage{latexsym}
\usepackage{mathrsfs}
\usepackage[all,cmtip]{xy}
\usepackage{enumitem}
\usepackage{xcolor}
\usepackage{graphicx}
\usepackage{comment}
\usepackage{mathabx,epsfig}
\usepackage{stmaryrd}
\usepackage{aliascnt}   
\usepackage{mathtools}

\usepackage[pdfborder={0 0 1},hypertexnames=false,hyperindex,breaklinks]{hyperref}

\makeatletter
\renewcommand\th@plain{\slshape}
\makeatother
\newtheoremstyle{plain}
 {2mm}
 {2mm}
 {\slshape}
 {}
 {\bfseries}
 {.}
 {.5em}
 {}
\theoremstyle{plain}
\newtheorem{theorem}{Theorem}[section]

\newaliascnt{corollary}{theorem}
\newtheorem{corollary}[corollary]{Corollary}
\aliascntresetthe{corollary}

\newaliascnt{lemma}{theorem}
\newtheorem{lemma}[lemma]{Lemma}
\aliascntresetthe{lemma}

\newaliascnt{proposition}{theorem}
\newtheorem{proposition}[proposition]{Proposition}
\aliascntresetthe{proposition}

\newaliascnt{claim}{theorem}

\aliascntresetthe{claim}

\newtheorem*{claim*}{Claim}

\newaliascnt{conjecture}{theorem}
\newtheorem{conjecture}[conjecture]{Conjecture}
\aliascntresetthe{conjecture}

\newaliascnt{question}{theorem}
\newtheorem{question}[question]{Question}
\aliascntresetthe{question}

\newtheoremstyle{definition}
 {2mm}
 {2mm}
 {\normalfont}
 {}
 {\bfseries}
 {.}
 {.5em}
 {}
\theoremstyle{definition}

\newaliascnt{definition}{theorem}
\newtheorem{definition}[definition]{Definition}
\aliascntresetthe{definition}

\newaliascnt{notation}{theorem}

\aliascntresetthe{notation}

\newaliascnt{remark}{theorem}
\newtheorem{remark}[remark]{Remark}
\aliascntresetthe{remark}

\newaliascnt{example}{theorem}

\aliascntresetthe{example}

\newtheorem*{acknowledgements}{Acknowledgements}

\usepackage[nameinlink]{cleveref}
\usepackage{autonum}

\crefname{section}{Section}{Sections}

\crefname{theorem}{Theorem}{Theorems}
\crefname{corollary}{Corollary}{Corollaries}
\crefname{lemma}{Lemma}{Lemmas}
\crefname{proposition}{Proposition}{Propositions}
\crefname{claim}{Claim}{Claims}
\crefname{definition}{Definition}{Definitions}
\crefname{notation}{Notation}{Notations}
\crefname{problem}{Problem}{Problems}
\crefname{note}{Note}{Notes}
\crefname{remark}{Remark}{Remarks}
\crefname{example}{Example}{Examples}
\crefname{conjecture}{Conjecture}{Conjectures}
\crefname{question}{Question}{Questions}

\crefname{enumi}{}{}
\crefname{enumii}{}{}
\crefname{enumiii}{}{}
\crefformat{equation}{{\upshape(#2#1#3)}}

\numberwithin{equation}{section}

\def\C{{\mathbb C}}
\def\Q{{\mathbb Q}}
\def\R{{\mathbb R}}
\def\Z{{\mathbb Z}}
\def\P{{\mathbb P}}
\def\A{{\mathbb A}}
\def\F{{\mathbb F}}

\def\QQ{\overline{\mathbb Q}}
\def\p{{ \mathfrak{p}}}

\def\O{{ \mathcal{O}}}
 
\def\a{{ \mathfrak{a}}}

\def\L{{ \mathcal{L}}}
\def\M{{ \mathcal{M}}}

\def\acts{\ \rotatebox[origin=c]{-90}{$\circlearrowright$}\ }
\def\racts{\ \rotatebox[origin=c]{90}{$\circlearrowleft$}\ }

\renewcommand{\mod}[1]{(\mathrm{mod}\ #1)}

\DeclareMathOperator{\pr}{pr}
\DeclareMathOperator{\Pic}{Pic}

\DeclareMathOperator{\id}{id}

\DeclareMathOperator{\Spec}{Spec}

\DeclareMathOperator{\Supp}{Supp}

\DeclareMathOperator{\Tr}{Tr}

\DeclareMathOperator{\Ima}{Im}
\DeclareMathOperator{\Per}{Per}
\DeclareMathOperator{\Fr}{Fr}
\DeclareMathOperator{\Gal}{Gal}
\DeclareMathOperator{\red}{red}

\usepackage{scalerel}
\DeclareMathAlphabet{\mathpzc}{OT1}{pzc}{m}{it}
\newcommand{\PerSch}{\scaleto{\mathpzc{Per}}{.76em}}

\renewcommand{\a}{\alpha}

\renewcommand{\b}{\beta}

\newcommand{\e}{\varepsilon}
\newcommand{\f}{\varphi}

\newcommand{\s}{\sigma}

\newcommand{\G}{\Gamma}

\allowdisplaybreaks

\begin{document}

	\title[Height boundedness of periodic and preperiodic points]
	{On the height boundedness of periodic and preperiodic points of dominant rational self-maps on projective varieties} 
	
	\author[Yohsuke Matsuzawa]{Yohsuke Matsuzawa}
	\address{Department of Mathematics, Graduate School of Science, Osaka Metropolitan University, 3-3-138, Sugimoto, Sumiyoshi, Osaka, 558-8585, Japan}
	\email{matsuzaway@omu.ac.jp}

    \author[Kaoru Sano]{Kaoru Sano}
    \address{NTT Institute for Fundamental Mathematics, NTT Communication Science Laboratories, NTT, Inc., 2-4, Hikaridai, Seika-cho, Soraku-gun, Kyoto 619-0237, Japan}
    \email{kaoru.sano@ntt.com}

	\begin{abstract} 
        We give a counterexample to the following conjecture: 
        the set of isolated periodic points of an automorphism of degree at least two on an affine space is a set of bounded height.
        As a positive result, we prove that any cohomologically hyperbolic dominant rational self-map on a projective variety admits a non-empty Zariski open subset on which the set of periodic points is height bounded.
		Concerning preperiodic points, we give an example suggesting 
        that the same statement may fail.
	\end{abstract}

	\subjclass[2020]{Primary 37P15; 
		Secondary 
		37P55 
	}	
	
	\keywords{Arithmetic dynamics, Periodic points, Height, cohomologically hyperbolic maps}

	\maketitle
	\tableofcontents

\section{Introduction}\label{sec:intro}
Let $f \colon \P^N_{\QQ} \longrightarrow \P^N_{\QQ}$
be a non-isomorphic surjective morphism defined over the field of
algebraic numbers $\QQ$.
It is well-known that the set of preperiodic points of $f$
in $\P^N(\QQ)$ is a set of bounded height.
Indeed, a point $x \in \P^N(\QQ)$ is preperiodic under $f$
if and only if $\hat{h}_f(x) = 0$, where $\hat{h}_f$
is the canonical height function of $f$.
Since $\hat{h}_f$ is equal to the naive height function up to a bounded function, we get the above-mentioned property.

It is natural to ask when the height boundedness of preperiodic points 
holds if we consider other types of self-maps on algebraic varieties.
For example, for automorphisms on $\A^N_{\QQ}$, there is the following conjecture.
\begin{conjecture}[{\cite[Conjecture 7.21]{silADS}}]\label{conj:htbddperautoaffine}
    Let $f \colon \A^N_{\QQ} \longrightarrow \A^N_{\QQ}$
    be an automorphism such that the maximum degree of the coordinate functions
    is at least two.
    Then the set of isolated periodic points of $f$ has bounded height.
\end{conjecture}
Here, a periodic point $P \in \A^N_{\QQ}$ is said to be isolated 
if it is not in the Zariski closure of $\{Q \in \A^N(\QQ) \mid f^n(Q)=Q\} \setminus \{P\}$
for all $n$. In particular, if the set of $n$-periodic points is finite
for all $n$, then all periodic points are isolated.
We also note that the degree assumption appears unnecessary.
Indeed, if $f$ is an automorphism defined by degree one polynomials,
then $f$ has only finitely many isolated periodic points.
Let us clarify the definition of height boundedness.

\begin{definition}
    Let $X$ be a quasi-projective variety over $\QQ$.
    A subset $S \subset X(\QQ)$ is said to be height bounded if the following holds.
    For an open immersion $i \colon X \hookrightarrow P$ into a projective variety $P$
    and a Weil height function $h_H$ associated with an ample divisor $H$
    on $P$,  the function $h_H \circ i$ is bounded on $S$.
\end{definition}
Note that this definition is independent of the choice of $i$, $H$, and $h_H$
(see for example \cite[Remark 2.2]{MX25}).

\cref{conj:htbddperautoaffine} is known for dimension two \cite{Den95}.
In higher dimensions, it holds for regular affine automorphisms
and triangular automorphisms \cite{Mar00, Mar03}.
See \cite[section 7.1.5]{silADS} for more details about this conjecture and related results.

In this paper, we give a counterexample to \cref{conj:htbddperautoaffine} in dimension three.

\begin{theorem}\label{thmintro:autoA3unbddper}
    Let 
    \begin{align}
        f \colon \A^3_{\Q} \longrightarrow \A^3_{\Q},
        (x,y,z) \mapsto (y, x+y-y^3, 2z-y).
    \end{align}
    Then, the following statements hold.
    \begin{enumerate}
        \item For any $n \geq 1$, the set of $n$-periodic points of $f$ is finite.
        \item The set of periodic points $\{P \in \A^3(\QQ) \mid \text{$f^n(P)=P$ for some $n \geq 1$}\}$ is height unbounded.
    \end{enumerate}
\end{theorem}
We actually prove that there is a sequence of periodic points that is Zariski dense, and
the height goes to infinity. See \cref{thm:autoA3unbddper}.
The idea of this construction is as follows.
The first two coordinates of our automorphism $f$ form an automorphism $g$
of $\A^2$, and it is an H\'enon map.
Moreover periodic points of $f$ have one-to-one correspondence 
to that of $g$ via projection, and $z$-coordinates have explicit expression 
in terms of the coordinates of $g$-periodic points.
We construct $f$ so that $g$ becomes a linear map after reduction mod $3$.
This structure provides enough information to show that the $z$-coordinates have a large height.

This counterexample naturally leads to the question of when the heights of (pre)periodic points are bounded.
It has turned out that cohomological hyperbolicity plays a key role in this question.

\begin{definition}
    A dominant rational map $f \colon X \dashrightarrow X$
    on a projective variety $X$ defined over an algebraically closed field 
    is said to be cohomologically hyperbolic if there is a unique 
    $p \in \{1, \dots, \dim X\}$ such that
    \begin{align}
        d_p(f) = \max\{d_i(f) \mid 1 \leq i \leq \dim X\}.
    \end{align}
    Here $d_i(f)$ denotes the $i$-th dynamical degree of $f$.
    In the above case, we say $f$ is $p$-cohomologically hyperbolic.
    See for example \cite{Da20,Tr20,DS05} for definitions and basic properties of 
    dynamical degrees.
\end{definition}

We note that the map $f$ in \cref{thmintro:autoA3unbddper} is not 
cohomologically hyperbolic. 
Indeed, we prove $d_1(f) = d_2(f) = 3$ in \cref{thm:autoA3unbddper}.
To state our next theorems, let us introduce some notation.

\begin{definition}
    Let $f \colon X \dashrightarrow X$ be a dominant rational map 
    on a quasi-projective variety $X$ defined over an algebraically closed field $k$.
    The set of $k$-points whose forward $f$-orbit is well-defined is denoted by
    \begin{align}
        X_f(k) \coloneqq \{P \in X(k) \mid \text{$f^n(P) \not\in I_f$ for all $n \geq 0$}\}
    \end{align}
    where $I_f$ is the indeterminacy locus of $f$.
    For a point $P \in X_f(k)$, its orbit is denoted by
    $O_f(P) \coloneqq \{f^n(P) \mid n \geq 0\}$.
\end{definition}

Wang and the first author proved the following height boundedness for cohomologically hyperbolic maps.

\begin{theorem}[Wang, Matsuzawa-Wang]\label{thm:WM}
    Let $f \colon X \dashrightarrow X$ be a cohomologically hyperbolic 
    dominant rational map on a smooth projective variety over $\QQ$.
    \begin{enumerate}
        \item {\rm (\cite[Theorem 1.1]{Wa22})} Suppose that $f$ is birational. Then there is a non-empty Zariski open subset
        $U \subset X$ such that the set
        \begin{align}
            \{ P \in X_f(\QQ) \cap X_{f^{-1}}(\QQ) \mid \text{$P$ is $f$-periodic and $O_f(P) \subset U$} \}
        \end{align}
        is height bounded.

        \item\label{MWdcohohyppreper} {\rm (\cite[Theorem F]{MW22})} Suppose that $f$ is $\dim X$-cohomologically hyperbolic.
        Then there is a non-empty Zariski open subset $U \subset X$ such that the set
        \begin{align}
            \{ P \in X_f(\QQ) \mid \text{$P$ is $f$-preperiodic and $O_f(P) \subset U$} \}
        \end{align}
        is height bounded.
    \end{enumerate}
\end{theorem}

\begin{remark}
    Note that since the statement of \cref{thm:WM}
    is about the existence of Zariski open subsets, the same conclusion holds after replacing $X$
    with a birationally equivalent variety (note that height boundedness is birationally invariant).
    Namely, \cref{thm:WM} holds for an arbitrary projective variety $X$.
\end{remark}

We prove that the same height boundedness of periodic points 
hold without assuming birationality.

\begin{theorem}\label{thm:perptishtbdd}
    Let $f \colon X \dashrightarrow X$ be a cohomologically hyperbolic map on a projective variety over $\QQ$. 
    Then there is a non-empty Zariski open subset $U \subset X$
    such that the set
    \begin{align}
        \{ P \in X_f(\QQ) \mid \text{$P$ is $f$-periodic and $O_f(P) \subset U$  }\} \label{setofperptinU}
    \end{align}
    is height bounded.
\end{theorem}
The idea and key ingredient of the proof are the same 
with previous works \cite{MW22,Wa22}, although we use terminologies
from \cite{xie2024algebraic} in this paper.
The key ingredient is the recursive inequality of heights along the orbit
that is proven in \cite{MW22} and \cite{xie2024algebraic}.
This proof also allows us to give an upper bound of
the heights of points of \cref{setofperptinU} in terms of ``height of the map $f$''.
See \cref{thm:easyuniformboundofhtofperpt} for the precise statement.

For the sake of consistency with the isolatedness assumption in 
\cref{conj:htbddperautoaffine}, we remark that it is known that periodic points of a cohomologically hyperbolic map are generically isolated.
More precisely, 
in the setting of \cref{thm:perptishtbdd}, there is a non-empty open subset $V \subset X$
such that every periodic point $P \in X_f(\QQ)$ satisfying $O_f(P) \subset V$ is
not contained in the closure of $\{Q \in X_f(\QQ) \mid f^n(Q)=Q \} \setminus \{P\}$
for all $n \geq 1$ (see \cite[Corollary 6.6]{xie2024algebraic}).

As for preperiodic points, we construct an example 
that suggests that the condition ``$\dim X$-cohomologically hyperbolic''
in \cref{thm:WM} (\cref{MWdcohohyppreper}) might be essential.

\begin{theorem}\label{thmintro:unbddpreper}
    Let $d \geq 3$ and $p$ be a prime number. Let
    \begin{align}
        f  \colon \P^2_{\QQ} \dashrightarrow \P^2_{\QQ}, (x:y:z) 
        \mapsto \bigg(x^d - \frac{1}{p^{d-3}}y^2z^{d-2} : xy^2z^{d-3} : y^2 z^{d-2}  \bigg).
    \end{align}
    \begin{enumerate}
        \item We have $d_1(f) = d$, $d_2(f)=2$. In particular, $f$ is $1$-cohomologically hyperbolic.
        \item Let $d = 4$, $p=2$. Let $P_0=(\frac{1+\sqrt{3}}{2}:\frac{1+\sqrt{3}}{2}:1)$.
        Then $f(P_0)=P_0$, and there is a sequence of points $\{P_n\}_{n \geq 1}$
        of $(\P^2 \setminus I_f)(\QQ)$
        such that 
        \begin{align}
            &f(P_n)= P_{n-1} \quad \text{for $n \geq 1$},\\
            &\text{$\{P_n \mid n \geq 0 \}$ is Zariski dense}, \text{ and}\\
            &\lim_{n \to \infty} h(P_n) = \infty,
        \end{align}
        where $h$ is  the naive height function on $\P^2(\QQ)$.
    \end{enumerate}
\end{theorem}
We also prove that for any $n \geq 0$, the map $f^n$ is finite over a neighborhood of $P_0$.
(See \cref{sec:unbddpreperpts} for more details.)
Thus, this height unboundedness seems to have nothing to do with
contracted curves. Note that if a self-map contracts a curve to a preperiodic point,
then it is obvious that the set of preperiodic points is height-unbounded.
We also remark that \cref{thmintro:unbddpreper} does not prove that the map is a counterexample to \cref{thm:perptishtbdd} when ``periodic'' is replaced by ``preperiodic'', although we expect it is.

\medskip
\noindent{\bf Future questions.}
\cref{conj:htbddperautoaffine} is originally proposed in \cite{Sil94}
in a slightly weaker form. It asked if the set of periodic points defined over a fixed number field is finite. 
By the Northcott property of height functions, height boundedness
is stronger than this finiteness over a number field.
Our example \cref{thmintro:autoA3unbddper} does not give a counterexample 
to this weaker version. Indeed, it is proven that periodic points of $f$
is determined by their first two coordinates, and the first two coordinates
form a periodic point of an H\'enon map. Thus, the first two coordinates are 
height-bounded and hence their degrees over $\Q$ go to infinity.

Regarding \cref{thm:perptishtbdd} for preperiodic points,
the following is still open.
\begin{question}
    Is there a cohomologically hyperbolic map $f \colon X \dashrightarrow X$
    on a projective variety $X$ over $\QQ$ with the following property?
    For any non-empty Zariski open subset $U \subset X$, the set
    \begin{align}
        \{ P \in X_f(\QQ) \mid \text{$P$ is $f$-preperiodic and $O_f(P) \subset U$  }\} 
    \end{align}
    is not height bounded.
\end{question}
We expect that the map in \cref{thmintro:unbddpreper} provides such an example.
Proving this or constructing other examples remains an interesting problem.
On the other hand, it could still be possible to expect that
there are only finitely many preperiodic points over a fixed number field.
At least, the sequence $\{P_n\}$ constructed in the proof of \cref{thmintro:unbddpreper} does not provide a counterexample, since it cannot be defined over a single number field (see \cref{prop:backorbitdeffldunbdd}).

In \cref{thm:WM} and \cref{thm:perptishtbdd}, we look at 
(pre)periodic points whose entire orbits are contained in the open set $U$.
This condition is essential in our proof, 
but it is not clear whether it reflects a genuine necessity of the statement
or merely a technical assumption imposed by our method. 
It is conceivable that the statement fails without this assumption.

We collect the preceding discussion into the following questions.
\begin{question}\label{que:variantsofhtbddness}
    Let $f \colon X \dashrightarrow X$ be a cohomologically hyperbolic map 
    on a (geometrically integral) projective variety defined over a number field $K$.
    \begin{enumerate}
        \item\label{queperptinopen} Is there a non-empty Zariski open subset $U \subset X$ such that the set
        \begin{align}
            \{P \in X_f(\overline{K}) \cap U(\overline{K})\mid \text{$P$ is $f$-periodic} \}
        \end{align}
        is height bounded?

        \item\label{quepreperptinopen} Suppose $f$ is $\dim X$-cohomologically hyperbolic. Then, is there a non-empty Zariski open subset $U \subset X$ such that the set
        \begin{align}
            \{P \in X_f(\overline{K}) \cap U(\overline{K})\mid \text{$P$ is $f$-preperiodic} \}
        \end{align}
        is height bounded?

        \item\label{quepreperptovernumberfield} Is there a non-empty Zariski open subset $U \subset X$ such that for 
        any finite extension $L/K$, the set
        \begin{align}
            \{P \in X_f(L) \mid \text{$P$ is $f$-preperiodic and $O_f(P) \subset U$} \}
        \end{align}
        or, more strongly, the set
        \begin{align}
            \{P \in X_f(L) \cap U(L) \mid \text{$P$ is $f$-preperiodic} \}
        \end{align}
        is finite?
    \end{enumerate}
\end{question}

\begin{remark}
    \cref{que:variantsofhtbddness} (\cref{quepreperptinopen})
    is not true if the assumption $\dim X$-cohomologically hyperbolic
    is replaced with simply cohomologically hyperbolic.
    Indeed, \cref{thmintro:unbddpreper} gives a counterexample.
\end{remark}

\medskip
\noindent{\bf Organization of the paper.}
In \cref{sec:autA3perptunbddht}, we prove \cref{thmintro:autoA3unbddper},
a counterexample to \cref{conj:htbddperautoaffine}.
In \cref{sec:htbddperpt}, we prove \cref{thm:perptishtbdd} and 
a uniform upper bound of the height of periodic points in terms of the height of the map.
In \cref{sec:unbddpreperpts}, we prove \cref{thmintro:unbddpreper}.

\medbreak
\noindent{\bf Notation.}

In this paper, we work over a number field or a field of characteristic zero unless otherwise stated.
\begin{itemize}
\item A  \emph{variety} over a field $k$ is a separated scheme of finite type over $k$ which is irreducible and reduced;
\item For a self-morphism $f \colon X \longrightarrow X$ of a scheme over $k$ and a 
point $P$ of $X$ (scheme point or $k'$-valued point where $k'$ is a field containing $k$), the  \emph{$f$-orbit of $P$}
is denoted by $O_{f}(P)$, i.e.\ $O_{f}(P) = \{ f^{n}(P) \mid n=0,1,2, \dots\}$.
The same notation is used for dominant rational map $f \colon X \dashrightarrow X$ on a variety $X$ defined over $k$
and $P \in X_{f}(k) = \{ P \in X(k) \mid f^{n}(P) \notin I_{f}, \ n \geq 0\}$
where $I_f$ is the indeterminacy locus of $f$.

\item
We set $M_{\Q}=\{|\ |_{p} \mid \text{$p=\infty$ or a prime number} \}$ with
\begin{align}
 &|a|_{\infty} =
 \begin{cases}
 a \quad \text{if $a\geq0$}\\
 -a \quad \text{if $a<0$}
 \end{cases}
 \\
 & |a|_{p} = p^{-n} \quad \txt{if $p$ is a prime and  $a=p^{n}\frac{k}{l}$ where\\ $k,l$ are non zero integers coprime to $p$.}
\end{align}
For a number field $K$, $M_{K}$ consists of the absolute values of $K$ extending
the absolute values in $M_\Q$.
Namely, $M_K$ consists of the absolute values
\begin{align}
|a|_{v} = |N_{K_{v}/\Q_{p}}(a)|_{p}^{1/[K_v:\Q_p]} ,\ a \in K
\end{align}
where $v$ is a place of $K$ which restricts to $p = \infty$ or a prime number.
Here, for an absolute value $v$ on $K$, $K_v$ denotes the completion of $K$ with respect to $v$.

\item For a Cartier divisor $D$ on a projective variety over $\QQ$,
$h_D$ denotes a logarithmic Weil height function associated with $D$.
We refer \cite{BG06, HS00, La83} for the basic theory of height functions.
\end{itemize}

\begin{acknowledgements}
    The authors thank Max Weinreich for pointing out a simpler method to compute dynamical degrees in \cref{sec:unbddpreperpts}.
    This work was done while the first author was on a research stay at Harvard University. He thanks their hospitality, especially Laura DeMarco, for hosting him.
    The first author is supported by JSPS KAKENHI Grant Numbers JP22K13903 and 23KK0252.
\end{acknowledgements}

\section{An automorphism of \texorpdfstring{$\A^3$}{A3} admitting periodic points with unbounded height}\label{sec:autA3perptunbddht}

In this section, we prove \cref{thmintro:autoA3unbddper}.
As mentioned in the introduction, we also calculate the dynamical degrees of $f$.
We use the following notation. For a self-morphism $f \colon X \longrightarrow X$
of a scheme over a field $k$ and a field extension $K/k$, we write 
\begin{align}
    &\Per_n(f, X(K)) = \{ P \in X(K) \mid f^n(P)=P\} \quad n \in \Z_{\geq 1}\\
    &\Per(f, X(K)) = \{ P \in X(K) \mid \text{$f^n(P)=P$ for some $n \geq 1$}\}.
\end{align}
Let us denote $h \colon \A^3(\QQ) \longrightarrow \R$ the restriction of the naive height function on $\P^3$ to $\A^3$.

\begin{theorem}\label{thm:autoA3unbddper}
    Let 
    \begin{align}
        f \colon \A^3_{\Q} \longrightarrow \A^3_{\Q},
        (x,y,z) \mapsto (y, x+y-y^3, 2z-y).
    \end{align}
    Then, the following statements hold.
    \begin{enumerate}
        \item\label{d1d2off} The morphim $f$ is an automorphism, and we have $d_1(f)=d_2(f)=3$.
        \item\label{nperptfin} for any $n \geq 1$, $\Per_n(f,\A^3(\QQ))$ is finite, thus all periodic points of $f$ are isolated.
        \item\label{unbddsetofperpt} there is a sequence $\{P_i\}_{i \geq 1}$ of points of  $\Per(f, \A^3(\QQ))$ such that
        \begin{align}
            \lim_{i \to \infty} h(P_i) = \infty
        \end{align}
        and $\{P_i \mid i\geq 1\}$ is Zariski dense in $\A^3_{\QQ}$.
    \end{enumerate}
\end{theorem}

We first note the following fibration structure:
\begin{equation}
    \begin{tikzcd}
        \A^3_{\Q} \arrow[r, "f"] \arrow[d,"\pr",swap] & \A^3_{\Q} \arrow[d,"\pr"] \\
        \A^2_{\Q} \arrow[r, "g", swap] & \A^2_{\Q}
    \end{tikzcd}
\end{equation}
where $\pr$ is the projection to $(x,y)$-coordinates and 
$g$ is the morphism defined by $g(x,y)=(y,x+y-y^3)$.
Note that $g$ is an H\'enon map.
We use this notation throughout this section.
We first observe the following.
\begin{lemma}\label{lem:corrofperpts}
    Let 
    \begin{align}
        \f \colon \Per(g,\A^2(\QQ)) \longrightarrow \Per(f, \A^3(\QQ)),
        (x_0,y_0) \mapsto (x_0,y_0,z_0)
    \end{align}
    where $z_0$ is defined as follows. Let $n$ be the exact period of $(x_0,y_0)$
    under $g$, and write $g^i(x_0,y_0) = (x_i,y_i)$ for $i=0,\dots, n-1$.
    Then we set
    \begin{align}
        z_0 = \frac{1}{2^n - 1} \sum_{i=0}^{n-1} 2^{n-1-i}y_i. \label{expressionofz0}
    \end{align}
    Then $\f$ is a well-defined bijection. 
    Moreover if $P \in \Per(g,\A^2(\QQ))$ has exact period $n$ under $g$, then so does $\f(P)$ under $f$.
\end{lemma}
\begin{proof}
    Let $(x_0,y_0) \in \Per_n(g, \A^2(\QQ))$.
    Write $g^i(x_0,y_0) = (x_i,y_i)$ for $i \geq 0$.
    Note that $(x_i,y_i) = (x_{i+n},y_{i+n})$.
    Then for any $z_0 \in \QQ$, we have 
    \begin{align}
        f^i(x_0,y_0,z_0) = (x_i,y_i, 2^{i}z_0 - \sum_{j=0}^{i-1}2^{i-1-j}y_j) \label{iterateoffcoordinate}
    \end{align}
    for $i \geq 0$.
    Thus, if we set $z_0$ as in \cref{expressionofz0}, then $(x_0,y_0,z_0)$
    is an $n$-periodic point of $f$.
    This fact guarantees the well-definedness of $\f$.
    Also, if $(x_0,y_0)$ has exact period $n$, then so does $(x_0,y_0,z_0)$.
    As the injectivity of $\f$ is clear, we only need to show the surjectivity of $\f$.
    Let $Q = (x_0,y_0,w_0) \in \Per(f, \A^3(\QQ))$.
    Let $N \geq 1$ be the exact period of $Q$ under $f$
    and $n$ be that of $(x_0,y_0)$ under $g$.
    Then by \cref{iterateoffcoordinate}, we have
    \begin{align}
        (x_0,y_0,w_0) = f^N(x_0,y_0,w_0)
        =(x_N,y_N, 2^{N}w_0 - \sum_{j=0}^{N-1}2^{N-1-j}y_j).
    \end{align}
    Noticing $y_i = y_{i+n}$ and $n \mid N$, we have
    \begin{align}
        w_0 &= \frac{1}{2^N - 1} \sum_{j=0}^{N-1} 2^{N-1-j} y_j
        =\frac{1}{2^N - 1}  \sum_{k=1}^{N/n} \sum_{(k-1)n \leq j <kn} 2^{N-1-j} y_j\\
        &=
        \frac{1}{2^N - 1}  \sum_{k=1}^{N/n} \sum_{(k-1)n \leq j <kn} 2^{N-1-(j-(k-1)n)}2^{-(k-1)n} y_{j-(k-1)n}\\
        &=\frac{1}{2^N - 1}  \sum_{k=1}^{N/n}  2^{-(k-1)n} 
        \sum_{j=0}^{n-1} 2^{N-1-j}y_j\\
        &=\frac{1}{2^N - 1} \frac{1-2^{-N}}{1-2^{-n}} \sum_{j=0}^{n-1} 2^{N-1-j}y_j\\
        &=\frac{1}{2^n - 1} \sum_{j=0}^{n-1} 2^{n-1-j}y_j.
    \end{align}
    Thus $Q = \f(x_0,y_0)$ and we are done.
\end{proof}

\begin{proof}[Proof of \cref{thm:autoA3unbddper} (\cref{d1d2off})(\cref{nperptfin})]
    Note that we have $d_1(f) = \lim_{n \to \infty}\deg(f^n)^{1/n}$,
    where $\deg f$ denotes the maximum of the total degrees of coordinate functions of $f$.
    (The same is true for $g$.)
    By direct calculation or the fact that $g$ is a regular affine automorphism 
    (i.e.\ indeterminacy loci of $g$ and $g^{-1}$ as rational maps on $\P^2$ do not intersect), we have $d_1(g) = 3$ (see for example \cite[Theorem 7.10 (a)]{silADS}).
    Thus $d_1(f) \geq d_1(g) = 3$.
    We clearly have $d_1(f) \leq 3$, and thus $d_1(f) = 3$.
    Since $f^{-1}(x,y,z) = (y-x+x^3, x, (z+x)/2)$, 
    same argument implies $d_1(f^{-1}) = 3$.
    As $d_2(f) = d_1(f^{-1})$, we are done.

    By \cref{lem:corrofperpts}, there is a bijection 
    $\Per_n(g, \A^2(\QQ)) \longrightarrow \Per_n(f,\A^3(\QQ))$.
    Since $g$ is a regular affine automorphism with $d_1(g) > 1$, 
    the set $\Per_n(g, \A^2(\QQ))$ is finite.
    (See \cite[Theorem 7.10(c)]{silADS}. It also follows from the fact
    that the set of periodic points of $g$ is height bounded \cite[Corollary B]{Ka06}, and the set of $n$-periodic points is a Zariski closed subset.)
\end{proof}

Our strategy to prove \cref{thm:autoA3unbddper} (\cref{unbddsetofperpt})
is as follows.
For a given $n$-periodic point $(x_0,y_0)$ of $g$, 
we set $z_0$ as \cref{expressionofz0}.
Then $(x_0,y_0,z_0)$ is an $n$-periodic point of $f$.
Consider a prime number $p$ such that a high power of $p$ divides $2^n-1$.
In this situation, if we can show that 
\begin{align}
    \left| \sum_{i=0}^{n-1} 2^{n-1-i}y_i\right|_v
\end{align}
is not small (e.g.\ $\geq 1$) for most of absolute values $v$
lying over $p$, then the height of $(x_0,y_0,z_0)$ would be large.
We will justify this argument for $p = 3$ and specific $n$'s.

We collect several properties of periodic points of $g$.
First, we see that the coordinates are algebraic integers.

\begin{lemma}\label{lem:integralityofperpt}
    Let $ (x_0,y_0) \in \A^2(\QQ)$ be a periodic point of $g$.
    Then $x_0,y_0$ are integral over $\Z$.
\end{lemma}
\begin{proof}
    Let $K \subset \QQ$ be a number field such that $x_0,y_0 \in K$.
    It is enough to show that for any non-archimedean 
    $v \in M_K$, we have $|x_0|_v, |y_0|_v \leq 1$.
    Write $g^i(x_0,y_0) = (x_i,y_i)$ for $i \in \Z$.
    By
    \begin{align}
        &(x_i,y_i) = (y_{i-1}, x_{i-1} + y_{i-1} - y_{i-1}^3)\\
        &(x_{i+1},y_{i+1}) = (y_{i}, x_{i} + y_{i} - y_{i}^3),
    \end{align}
    we have
    \begin{align}
        y_{i+1} = y_{i-1} + y_i - y_i^3
    \end{align}
    for all $i \in \Z$.
    Let $i_0 \in \Z$ be such that $|y_{i_0}|_v = \max\{|y_i|_v \mid i\in \Z\}$.
    (Such $i_0$ exists because $(x_0,y_0)$ is $g$-periodic.)
    If $|y_{i_0}|_v > 1$, we have
    \begin{align}
        |y_{i_0 + 1}|_v = |y_{i_0-1} + y_{i_0} - y_{i_0}^3  |_v = |y_{i_0}^3|_v > |y_{i_0}|_v
    \end{align}
    and this is a contradiction.
    Thus $|y_i|_v \leq 1$ for all $i \in \Z$.
    In particular, $|y_0|_v \leq 1$ and $|x_0|_v = |y_{-1}|_v \leq 1$, and we are done.
\end{proof} 

To state some algebraic-geometric properties, we define the schemes corresponding to the set of periodic points.
\begin{definition}
    For a positive integer $l$, we write 
    $g^l(x,y) = (X_l(x,y), Y_l(x,y))$. Note that $X_l(x,y), Y_l(x,y)$
    are polynomials in $x,y$ with coefficients in $\Z$.
    We set 
    \begin{align}
        \PerSch_l \coloneqq \Spec \Z[x,y]/(X_l - x, Y_l - y)  
    \end{align}
    and let $\pi_l \colon \PerSch_l \longrightarrow \Spec \Z$
    be the structure morphism.
    Note that for any ring $R$, the set of $R$-valued points
    $\PerSch_l(R)$ as a subset of $\A^2(R)$ is exactly the set of
    $l$-periodic points of $g$ 
    ($g$ acts on $\A^2(R)$ since it is defined over $\Z$).
\end{definition}

\begin{lemma}\label{lem:perschetaleover3}
    Let $l$ be a positive integer such that $l \not\equiv 0 \ \mod 8$.
    Then for any point $\xi \in \PerSch_l$ such that $\pi_l(\xi) = (3)$,
    $\pi_l$ is \'etale at $\xi$. 
\end{lemma}
\begin{proof}
    Let $\p \subset \Z[x,y]/(X_l - x, Y_l - y)=:R$ be the prime ideal
    corresponding to $\xi$.
    Note that $\p \cap \Z = (3)$.
    Let 
    \begin{align}
        A = 
        \begin{pmatrix}
            \frac{\partial X_l}{\partial x} - 1 & \frac{\partial X_l}{\partial y}\\
            \frac{Y_l}{\partial x} & \frac{\partial Y_l}{\partial y} - 1
        \end{pmatrix}.
    \end{align}
    By the Jacobian criterion (see ,for example, \cite[Corollary 2.5.9]{Fuetale}), it is enough to show that
    the image of $\det A$ in $R$ is not contained in $\p$.
    To this end, we show that the image of $\det A$ in $\F_3[x,y]$
    is invertible.
    First note that the Jacobian matrix $Jg$ of $g$ is
    \begin{align}
        \begin{pmatrix}
            0 & 1\\
            1 & 1 - 3y^2
        \end{pmatrix}
        \equiv
        \begin{pmatrix}
            0 & 1\\
            1 & 1
        \end{pmatrix}
        \ \mod 3.
    \end{align}
    Thus we have
    \begin{align}
        A \equiv 
        {
        \begin{pmatrix}
            0 & 1\\
            1 & 1
        \end{pmatrix}
        }^l 
        -
        \begin{pmatrix}
            1 & 0\\
            0 & 1
        \end{pmatrix}
        \ \mod 3.
    \end{align}
    We can check that the matrix
    $
    \begin{pmatrix}
        0 & 1\\
        1 & 1
    \end{pmatrix}
    $
    has order $8$ modulo $3$, 
    and the image of $\det A$ in $\F_3[x,y]$ are invertible unless $8 \mid l$
    (see the following table).

\[
\setlength{\arraycolsep}{3pt} 
\begin{array}{c|cccccccc}
l 
& 1 & 2 & 3 & 4 & 5 & 6 & 7 & 8 \\ \hline 
{\begin{pmatrix}0 & 1\\ 1 & 1\end{pmatrix}}^l \rule[-14pt]{0pt}{35pt}
& \begin{pmatrix}0 & 1\\ 1 & 1\end{pmatrix}
& \begin{pmatrix}1 & 1\\ 1 & 2\end{pmatrix}
& \begin{pmatrix}1 & 2\\ 2 &0 \end{pmatrix}
& \begin{pmatrix}2 & 0\\ 0 & 2\end{pmatrix}
& \begin{pmatrix}0 & 2\\ 2 & 2\end{pmatrix}
& \begin{pmatrix}2 & 2\\ 2 & 1\end{pmatrix}
& \begin{pmatrix}2 & 1\\ 1 & 0\end{pmatrix}
& \begin{pmatrix}1& 0\\ 0 & 1\end{pmatrix}
\\ \hline
A \mod 3 \rule[-14pt]{0pt}{35pt}
& \begin{pmatrix}-1 & 1\\ 1 & 0\end{pmatrix}
& \begin{pmatrix}0 & 1\\ 1 & 1\end{pmatrix}
& \begin{pmatrix}0 & 2\\ 2 & -1\end{pmatrix}
& \begin{pmatrix}1 & 0\\ 0 & 1\end{pmatrix}
& \begin{pmatrix}-1 & 2\\ 2 & 1\end{pmatrix}
& \begin{pmatrix}1 & 2\\ 2 & 0\end{pmatrix}
& \begin{pmatrix}1 & 1\\ 1 & -1\end{pmatrix}
& 0
\end{array}
\]
\end{proof}

\begin{corollary}\label{cor:redisinjective}
    Let $l$ be a positive integer such that $l \not\equiv 0 \ \mod 8$.
    Let $K \subset \QQ$ be a number field with the ring of integers $\O_K$.
    Let $\p \subset \O_K$ be a prime ideal lying over $3$.
    Then the map induced by the reduction 
    \begin{align}
        \PerSch_l(\O_K) \longrightarrow \A^2(\O_K/\p) 
    \end{align}
    is injective.
\end{corollary}
\begin{proof}
    Let $P,Q \in \PerSch_l(\O_K)$ be such that $\overline{P} = \overline{Q}$,
    where $\overline{\ \cdot\  }$ stands for the reduction.
    By \cref{lem:perschetaleover3}, 
    the base change $\pi_{l,\O_K} \colon (\PerSch_l)_{\O_K} \longrightarrow \Spec \O_K$ of $\pi_l$ is \'etale at any points over $\p$.
    If we regard $P,Q$ as sections of $\pi_{l,\O_K}$,
    we have $P(\p)=Q(\p)$ and $\pi_{l,\O_K}$ is \'etale at this point.
    Then, $P$ is an open immersion on a neighborhood of $\p$
    (see for example \cite[Proposition 2.3.1(iv), Proposition 2.3.9]{Fuetale}),
    and this implies $\pi_{l,\O_K}$ is open immersion around the point $P(\p)=Q(\p)$.
    Thus $P$ and $Q$ agrees around $\p$ and hence $P = Q$. 

\end{proof}

\begin{lemma}\label{lem:redcoordextdegm}
    Let $r$ be a non-negative integer, and set $n \coloneqq 2m$ and $m \coloneqq 3^r$.
    Let $P = (x_0,y_0) \in \A^2(\QQ)$ be an $g$-periodic point of period $n$.
    Let $K \subset \QQ$ be any number field containing $x_0,y_0$.
    By \cref{lem:integralityofperpt}, $x_0, y_0 \in \O_K$ (where $\O_K$ is the ring of integers of $K$). Let $\p \subset \O_K$ be a prime ideal lying over $3$.
    Then we have the following diagram
    \begin{equation}
        \begin{tikzcd}[row sep=1.2ex]
            g \acts \A^2(\O_K) \arrow[r] & \A^2(\O_K/\p) \racts \overline{g}\\
            P \arrow[r,|->] \arrow[u,phantom, "\in", sloped]& \overline{P}\arrow[u,phantom, "\in", sloped]
        \end{tikzcd}
    \end{equation}
    where $\overline{g}$, $\overline{P}$ are reduction of $g$, $P$
    modulo $\p$ respectively.
    Let $\overline{x_0}, \overline{y_0}$ be the image of $x_0,y_0$ in $\O_K/\p$.
    Then we have $[\F_3(\overline{x_0}, \overline{y_0}) : \F_3] \leq m$.
    Here $\F_3(\overline{x_0}, \overline{y_0}) $ is the subfield of $\O_K/\p$
    generated by $\overline{x_0}, \overline{y_0}$ over $\F_3$.
\end{lemma}
\begin{proof}
    Note that $\overline{g} \colon \A^2(\O_K/\p) \longrightarrow \A^2(\O_K/\p)$
    is an $\F_3$-linear map.
    We set
    \begin{align}
        D \colon \O_K/\p \longrightarrow \O_K/\p, a \mapsto \Fr_3(a) - a = a^3 -a,
    \end{align}
    where $\Fr_3$ is the $3$-power Frobenius map.
    This map is an $\F_3$-linear map.
    Using this, the map $\overline{g}$ can be expressed 
    by the following matrix:
    \begin{align}
        L\coloneqq
        \begin{pmatrix}
            0 & 1 \\
            1 & -D
        \end{pmatrix},
    \end{align}
    where we regard the entries as elements of the subring of the ring of $\F_3$-linear endomorphisms on $\O_K/\p$ generated by $D$.
    Then, we have
    \begin{align}
        &L^2 = 
        \begin{pmatrix} 1 & 0 \\ 0 & 1\end{pmatrix}
        -
        \begin{pmatrix} D & 0 \\ 0 & D \end{pmatrix} L,\\
        &L^n = (L^2)^{3^r} = 
        \begin{pmatrix} 1 & 0 \\ 0 & 1\end{pmatrix}
        -
        \bigg(\begin{pmatrix} D & 0 \\ 0 & D \end{pmatrix} L \bigg)^{3^r}
        =
        \begin{pmatrix} 1 & 0 \\ 0 & 1\end{pmatrix}
        -
        \begin{pmatrix} D^m & 0 \\ 0 & D^m \end{pmatrix} L^m. \label{nthiterateofL}
    \end{align}
    Now, since $\overline{P}$ has period $n$ under $\overline{g}$,
    we have
    \begin{align}
        L^n \begin{pmatrix}\overline{x_0} \\ \overline{y_0}\end{pmatrix}
        =
        \begin{pmatrix}\overline{x_0} \\ \overline{y_0}\end{pmatrix},
    \end{align}
    or equivalently
    \begin{align}
        \begin{pmatrix} D^m & 0 \\ 0 & D^m \end{pmatrix} L^m
        \begin{pmatrix}\overline{x_0} \\ \overline{y_0}\end{pmatrix}=0.
    \end{align}
    Since $L$ and $\begin{pmatrix} D & 0 \\ 0 & D \end{pmatrix}$
    commute and $L$ is an automorphism, we have
    $D^m(\overline{x_0}) = D^m(\overline{y_0})=0$.
    As $D = \Fr_3 - \id$.,
    we have $D^m = D^{3^r} = \Fr_3^m - \id$.
    Thus we have $\overline{x_0}^{3^m} = \overline{x_0}$
    and $\overline{y_0}^{3^m} = \overline{y_0}$, that is 
    $\overline{x_0}, \overline{y_0} \in \F_{3^m}$.
    (Here $\F_{3^m}$ is the one in a fixed algebraic closure of $\O_K/\p$.)
    Thus $[\F_3(\overline{x_0}, \overline{y_0}) : \F_3] \leq m$.
\end{proof}

\begin{lemma}\label{lem:fldgenbycoordofnperpts}
    Let $r$ be a non-negative integer, and set $n \coloneqq 2m$ and $m\coloneqq 3^r$.
    Then, the following statements hold.
    \begin{enumerate}
        \item $\# \PerSch_n(\QQ) = 3^n$.
        \item Let $K$ be the number field generated over $\Q$ by coordinates of all $n$-periodic points of $g$ defined over $\QQ$. 
        Let $\p \subset \O_K$ be any prime ideal lying over $3$.
        Then $[\O_K/\p : \F_3] = m$. Moreover, the map induced by the reduction
        \begin{align}
            \PerSch_n(\O_K) \longrightarrow \A^2(\O_K/\p)
        \end{align}
        is bijective.
        \item The field extension $K/\Q$ is Galois and unramified over $3$.
    \end{enumerate}
\end{lemma}
\begin{proof}
    (1)
    First, note that the inequality $\#\PerSch_n(\QQ) \leq 3^n$ holds.
    This inequality follows, for example, from \cite[Theorem 7.10 (c)]{silADS}.
    But in our situation, we can give direct proof as follows.
    Write $\PerSch_n(\QQ) = \{ P_1, P_2,\ldots, P_k\}$,
    and let $K'$ be the number field generated by the coordinates of these points.
    Let $\p \subset \O_{K'}$ be a prime ideal lying over $3$.
    By \cref{lem:integralityofperpt}, all $P_i$'s are in $\PerSch_n(\O_{K'})$,
    and by \cref{cor:redisinjective}, the reductions of $P_i$'s are distinct.
    But \cref{lem:redcoordextdegm} implies those reductions are contained in
    $\A^2(\F_{3^m})$, which has only $3^{2m}=3^n$ points.
    Hence, we have $\# \PerSch_n(\QQ) = k \leq 3^n$.

    Now, let us consider an unramified extension $F/\Q_3$ of degree $m$
    corresponding to residue field extension $\F_{3^m}/\F_3$.
    Let $\O_F$ be the ring of integers of $F$.
    Then note that $3\O_F$ is the maximal ideal of $\O_F$
    and $\O_F/3\O_F = \F_{3^m}$.
    Consider the following diagram
    \begin{equation}
        \begin{tikzcd}
            \PerSch_n  \arrow[d, swap, "\pi_n"] & (\PerSch_n)_{\O_F} \arrow[l] \arrow[d, "(\pi_n)_{\O_F}"] &  (\PerSch_n)_{\O_F/3\O_F} \arrow[l] \arrow[d]\\
            \Spec \Z & \Spec \O_F \arrow[l] & \Spec \O_F/3\O_F \arrow[l]
        \end{tikzcd}
    \end{equation}
    where all the squares are fiber products.
    We note that $(\PerSch_n)_{\O_F/3\O_F}(\O_F/3\O_F) = \A^2(\O_F/3\O_F)$.
    Indeed, the action of $g$ on $\A^2(\O_F/3\O_F)$ is represented by the same matrix 
    $L$ as in the proof of \cref{lem:redcoordextdegm}.
    As $\O_F/3\O_F = \F_{3^m}$, \cref{nthiterateofL} implies $L^n$
    is identity on $\A^2(\O_F/3\O_F)$, hence $(\PerSch_n)_{\O_F/3\O_F}(\O_F/3\O_F) = \A^2(\O_F/3\O_F)$.
    
    Next, since $n \not\equiv 0 \ \mod 8$,
    the morphism $(\pi_n)_{\O_F}$ is \'etale at any point of $(\PerSch_n)_{\O_F/3\O_F}$ by \cref{lem:perschetaleover3}.
    Therefore, by infinitesimal lifting property of \'etale morphisms (e.g.\ \cite[Theorem 2.6.2]{Fuetale}),
    every point of $(\PerSch_n)_{\O_F/3\O_F}(\O_F/3\O_F)$
    uniquely lifts to a point of $(\PerSch_n)_{\O_F}(\O_F)$
    (lift to a $\O_F/3^k\O_F$-point inductively using infinitesimal lifting property).
    Thus, we have
    \begin{align}
        3^n = \# \A^2(\O_F/3\O_F) = \#\PerSch_n(\O_F) \leq \#\PerSch_n(F).
    \end{align}
    Since $\PerSch_n$ is finite type over $\Z$ and $\# \PerSch_n(\QQ) < \infty$,
    $(\PerSch_n)_\Q$ is finite over $\Q$.
    Thus if we fix algebraic closure $\overline{F}$ of $F$ and let
    $\Omega \subset \overline{F}$ be the algebraic closure of $\Q$ in $\overline{F}$,
    we have $\PerSch_n(F) \subset \PerSch_n(\Omega)$, where both are considered to be subsets of 
    $\PerSch_n(\overline{F})$.
    Therefore, we get $\# \PerSch_n(\QQ) = \# \PerSch_n(\Omega) \geq \# \PerSch_n(F) \geq 3^n$,
    and hence $\# \PerSch_n(\QQ) = 3^n$.
    Note that this argument shows $\PerSch_n(\O_F) = \PerSch_n(\Omega)$ as subsets of $\PerSch_n(\overline{F})$.

    (3) Let $K \subset \QQ$ be the number field generated by all the coordinates of $n$-periodic points. Since Galois conjugate of $n$-periodic point is again $n$-periodic point,
    $K$ is Galois over $\Q$.
    The image of any embedding of $K$ into $\Omega$ is contained in $F$ by the above argument.
    Hence $K/\Q$ is unramified over $3$.
    
    (2) Let $\p \subset \O_K$ be any prime ideal lying over $3$.
    Then there is a field embedding $i \colon K \longrightarrow F$ such that 
    $i^{-1}(3\O_F) = \p$.
    This induces an embedding $\O_K/\p \subset \O_F/3\O_F = \F_{3^m}$.
    Then the composite map
    \begin{align}
        \PerSch_n(\O_K) \longrightarrow \A^2(\O_K/\p) \longrightarrow \A^2(\O_F/3\O_F )
    \end{align}
    is injective by \cref{cor:redisinjective}.
    Since we have $\# \PerSch_n(\O_K)=\# \A^2(\O_F/3\O_F) = 3^n$, both maps are bijective. 
    In particular, we have $[\O_K/\p: \F_3]=m$.
\end{proof}

{\bf Notation.} For $n\geq 1$, let $K_n$ denote the number field generated over $\Q$
by all the coordinates of $n$-periodic points of $g$ defined over $\QQ$.
Its ring of integers is denoted by $\O_{K_n}$.

\begin{lemma}\label{lem:reductionistrace}
    Let $r$ be a non-negative integer, and set $n \coloneqq 2m$ and $m \coloneqq 3^r$.
    Let $(x_0,y_0) \in \PerSch_n(K_n)$ and write $g^i(x_0,y_0)=(x_i,y_i)$
    for $i \geq 0$.
    Let $\p \subset \O_{K_n}$ be any prime ideal lying over $3$.
    Let
    \begin{align}
        \zeta \coloneqq \sum_{i=0}^{n-1} 2^{n-1-i} y_i.
    \end{align}
    (Recall that this is the ``numerator'' of $z$-coordinate of the periodic point of $f$
    corresponding to $(x_0,y_0)$, see \cref{expressionofz0}.)
    Then $x_i, y_i, \zeta \in \O_{K_n}$ by \cref{lem:integralityofperpt}, 
    and
    \begin{align}
        \overline{\zeta}=\sum_{i=0}^{n-1} (-1)^{n-1-i}\overline{y_i} 
        = \Tr_{\F_\p / \F_3}(\overline{x_0} - \overline{y_0})
    \end{align}
    where $\F_\p = \O_{K_n}/\p$,
    the symbol $\overline{\ \cdot \ }$ stands for the image in $\F_\p$,
    and $\Tr_{\F_\p / \F_3}$ is the trace of the field extension $\F_\p / \F_3$.
    In particular, the element $\overline{\zeta}$ is invariant up to sign if we replace 
    $(x_0,y_0)$ with $g^i(x_0,y_0)=(x_i,y_i)$ for any $i \geq 0$.
\end{lemma}
\begin{proof}
    Since $\p$ is lying over $3$ and $n$ is even, we have
    \begin{align}
        \overline{\zeta} = \sum_{i=0}^{n-1} (-1)^{n-1-i}\overline{y_i} 
        = - \sum_{i=0}^{n-1} (-1)^{i}\overline{y_i} 
    \end{align}
    Let $D \colon \F_\p \longrightarrow \F_p, a \mapsto a^3 - a$ and
    \begin{align}
        L = 
        \begin{pmatrix}
            0 & 1 \\ 1 & - D
        \end{pmatrix}
    \end{align}
    be as in the proof of \cref{lem:redcoordextdegm}.

    Consider the polynomial $P(t) = \sum_{i=0}^{n-1}(-t)^i \in \F_3[t]$.
    Then the second coordinate of
    \begin{align}
        P(L)
        \begin{pmatrix}
            \overline{x_0} \\ \overline{y_0}
        \end{pmatrix}
    \end{align}
    is $-\overline{\zeta}$.
    Recalling $n = 2m = 2\cdot 3^r$, we have
    \begin{align}
        P(t) & = \frac{1-t^n}{1+t} = \frac{(1-t^m)(1+t^m)}{1+t} 
        = \frac{(1-t^{3^r})(1+t^{3^r})}{1+t} = \frac{(1-t)^{3^r}(1+t)^{3^r}}{1+t}\\
        & = (1-t)^m(1+t)^{m-1} = (1-t)(1-t^2)^{m-1}.
    \end{align}
    Thus 
    \begin{align}
        P(L) 
        &= (1-L)(1-L^2)^{m-1} 
        =
        \bigg( 
            \begin{pmatrix}
                1 & -1 \\ -1 & 1
            \end{pmatrix}    
            +
            \begin{pmatrix}
                0 & 0 \\ 0 & D
            \end{pmatrix}   
        \bigg) 
            L^{m-1}
            \begin{pmatrix}
                D^{m-1} & 0 \\ 0 & D^{m-1} 
            \end{pmatrix}\\
        &=
        \bigg( 
            \begin{pmatrix}
                1 & -1 \\ -1 & 1
            \end{pmatrix}    
            +
            \begin{pmatrix}
                0 & 0 \\ 0 & D
            \end{pmatrix}   
        \bigg) 
        \bigg(
            \begin{pmatrix}
                0 & 1 \\ 1 & 0
            \end{pmatrix}
            -
            \begin{pmatrix}
                0 & 0 \\ 0 & D
            \end{pmatrix}
        \bigg)^{m-1}
        \begin{pmatrix}
            D^{m-1} & 0 \\ 0 & D^{m-1} 
        \end{pmatrix}.
    \end{align}
    By \cref{lem:fldgenbycoordofnperpts}, $[\F_\p:\F_3] = m$,
    and hence $D^m(\overline{x_0}) = D^m(\overline{y_0})=0$.
    Thus, we can calculate 
    \begin{align}
        P(L)
        \begin{pmatrix}
            \overline{x_0} \\ \overline{y_0}
        \end{pmatrix}
        =
        \begin{pmatrix}
            1 & -1 \\ -1 & 1
        \end{pmatrix}
        \begin{pmatrix}
            0 & 1 \\ 1 & 0
        \end{pmatrix}^{m-1}
        \begin{pmatrix}
            D^{m-1}(\overline{x_0}) \\ D^{m-1}(\overline{y_0})
        \end{pmatrix}
        =
        \begin{pmatrix}
            D^{m-1}(\overline{x_0}) -  D^{m-1}(\overline{y_0})\\
            -D^{m-1}(\overline{x_0}) + D^{m-1}(\overline{y_0}).
        \end{pmatrix}
    \end{align}
    Since $D = \Fr_3 - \id$, we have
    \begin{align}
        D^{m-1} = (\Fr_3 - \id)^{3^r - 1} = 1 + \Fr_3 + \cdots + \Fr_3^{m-1}.
    \end{align}
    Note that $(t-1)^{m-1} = 1 + t + \cdots +t^{m-1}$ in $\F_3[t]$.
    Thus, we conclude that $D^{m-1}(a) = \Tr_{\F_\p / \F_3}(a)$ holds for all $a \in \F_\p$,
    and therefore $-\overline{\zeta} = - \Tr_{\F_\p / \F_3}( \overline{x_0} - \overline{y_0})$.
\end{proof}

\begin{proposition}\label{prop:existenceperptwithhtbd}
    Let $r$ be a non-negative integer, and set $n \coloneqq 2m$ and $m \coloneqq 3^r$.
    Then, there is $(x_0,y_0) \in \PerSch_n(\O_{K_n})$ with the following property:
    \begin{align}
        \# \big\{ v \in M_{K_n} \bigm| v \mid 3,\ \Tr_{\F_{\p_v}/\F_3}( x_0 - y_0 \ \mod {\p_v}) = 0 \big\}
        \leq \frac{1}{3} \# \big\{v \in M_{K_n} \bigm| v \mid 3\big\}, \label{proportionoftracezero}
    \end{align}
    where $\p_v \subset \O_{K_n}$ is the prime ideal corresponding to the place $v$,
    and $\F_{\p_v} = \O_{K_n}/\p_v$ is the residue field at $\p_v$.
    Moreover, for such a point $(x_0,y_0)$, the corresponding $f$-periodic point $\f(x_0,y_0)$ satisfies
    \begin{align}
        h(f^i(\f(x_0,y_0))) \geq \frac{2}{3} (r+1) \log 3 \label{htlowerboundperpt}
    \end{align}
    for any $i \geq 0$.
    (Recall that $h$ is the height on $\A^3$, and $\f$ is the map defined in \cref{lem:corrofperpts}.)
\end{proposition}
\begin{proof}
    Let $G = \Gal (K_n/\Q)$ be the Galois group of the Galois extension $K_n / \Q$.
    Note that $G$ acts on $\PerSch_n(\O_{K_n})$ since $\PerSch_n$ is a scheme over $\Z$.
    Let
    \begin{align}
        \PerSch_n(\O_{K_n}) = \coprod_{i=1}^r O_i
    \end{align}
    be the $G$-orbit decomposition.
    Let us fix a prime ideal $\p_0 \subset \O_{K_n}$ lying over $3$.
    Then by \cref{lem:fldgenbycoordofnperpts}, the reduction map
    \begin{align}
        \red \colon \PerSch_n(\O_{K_n}) \longrightarrow \A^2(\O_{K_n}/\p_0) \simeq \F_{\p_0}^2
    \end{align}
    is bijective.
    Set 
    \begin{align}
        B = \{ (a,b) \in \F_{\p_0}^2 \mid \Tr_{\F_{\p_0}/\F_3}(a-b)=0 \}.
    \end{align}
    Since the map $\F_{\p_0}^2 \longrightarrow \F_3, (a,b) \mapsto \Tr_{\F_{\p_0}/\F_3}(a-b)$
    is an $\F_3$-linear surjection,we have $\# B = \# \F_{\p_0}^2 / 3$.
    Thus, there is $i \in \{1,\dots , r\}$ such that $\# \red(O_i) \cap B \leq \# \red(O_i) / 3$,
    that is 
    \begin{align}
        \# \{(\a , \b) \in O_i \mid \Tr_{\F_{\p_0}/\F_3}( \a - \b \ \mod {\p_0} ) = 0\}
        \leq \frac{1}{3} \# O_i . \label{boundoftracezeropts}
    \end{align} 
    Pick a point $(x_0,y_0) \in O_i$ for such $i$.
    We prove that this point satisfies the desired properties.
    As $K_n/\Q$ is Galois, th action of $G$ on $\{\p_v \mid v \in M_{K_n}, v \mid 3\}$ is transitive.
    Thus, we have
    \begin{align}
        &\#\big\{ v \in M_{K_n} \bigm| v \mid 3, \Tr_{\F_{\p_v}/\F_3}( x_0 - y_0 \ \mod {\p_v}) = 0 \big\}\\
        &=
        \# \{\s(\p_0) \mid \s \in G, \Tr_{\F_{\s(\p_0)}/\F_3}( x_0 - y_0 \ \mod {\s(\p_0)}) = 0  \}\\
        &=
        \# \{\s(\p_0) \mid \s \in G, \Tr_{\F_{\p_0}/\F_3}( \s^{-1}(x_0) - \s^{-1}(y_0) \ \mod {\p_0}) = 0  \}.
    \end{align}
    By \cref{boundoftracezeropts}, we obtain the inequality
    \begin{align}
        \# \{\s \in G \mid \Tr_{\F_{\p_0}/\F_3}( \s^{-1}(x_0) - \s^{-1}(y_0) \ \mod {\p_0}) = 0 \}
        \leq
        \frac{1}{3} \# G.
    \end{align}
    Moreover, if $H \subset G$ is the stabilizer of $\p_0$, 
    then $\Tr_{\F_{\p_0}/\F_3}( \s^{-1}(x_0) - \s^{-1}(y_0) \ \mod {\p_0})$
    is constant on each left coset of $H$. Thus, we conclude the inequalities
    \begin{align}
        &\# \{\s(\p_0) \mid \s \in G, \Tr_{\F_{\p_0}/\F_3}( \s^{-1}(x_0) - \s^{-1}(y_0) \ \mod {\p_0}) = 0  \}\\
        &\leq
        \frac{1}{3} \# G/H = \frac{1}{3} \# \big\{v \in M_{K_n} \bigm| v \mid 3\big\},
    \end{align}
    which proves \cref{proportionoftracezero}.

    Next, we prove the lower bound on the height.
    Note that $f^i(\f(x_0,y_0)) = \f(g^i(x_0,y_0))$.
    By the last assertion of \cref{lem:reductionistrace}, 
    each $g^i(x_0,y_0)$ also satisfies the property we just proved.
    Thus, it is enough to prove \cref{htlowerboundperpt} for $i=0$.

    Recall $\f(x_0,y_0)=(x_0,y_0,z_0)$ with
    \begin{align}
        z_0 = \frac{\zeta}{2^n - 1}, \quad \zeta \coloneqq \sum_{i=0}^{n-1} 2^{n-1-i}y_i,
    \end{align}
    where $(x_i,y_i)\coloneqq g^i(x_0,y_0)$.
    Thus, the coordinate $z_0$ is in $K_n$ as well, and
    \begin{align}
        h(\f(x_0,y_0)) &= h(x_0,y_0,z_0) 
        = \sum_{v \in M_{K_n}} 
        \frac{[(K_n)_v : \Q_v]}{[K_n:\Q]} \log\max\{|x_0|_v, |y_0|_v , |z_0|_v, 1 \}\\
        &\geq
        \sum_{\substack{v \in M_{K_n} \\ v \mid 3 }}
        \frac{[(K_n)_v : \Q_3]}{[K_n:\Q]} \log\max\{|z_0|_v, 1 \},
    \end{align}
    where $\Q_v$ is the completion of $\Q$ with respect to the restriction of $v$ to $\Q$.
    For any $v \in M_{K_n}$ with $v \mid 3$, $|\zeta|_v=1$ if 
    $\Tr_{\F_{\p_v}/\F_3}(x_0 - y_0 \ \mod {\p_v}) \neq 0$
    by \cref{lem:reductionistrace}.
    Thus
    \begin{align}
        &\sum_{\substack{v \in M_{K_n} \\ v \mid 3 }}
        \frac{[(K_n)_v : \Q_3]}{[K_n:\Q]} \log\max\{|z_0|_v, 1 \}\\
        &\geq
        \sum_{\substack{v \in M_{K_n} ,\  v \mid 3 \\ \Tr_{\F_{\p_v}/\F_3}(x_0 - y_0 \ \mod {\p_v}) \neq 0 }}
        \frac{[(K_n)_v : \Q_3]}{[K_n:\Q]} \log\max\{|z_0|_v, 1 \}\\
        &=
        \log\frac{1}{|2^n - 1|_3}
        \sum_{\substack{v \in M_{K_n} ,\  v \mid 3 \\ \Tr_{\F_{\p_v}/\F_3}(x_0 - y_0 \ \mod {\p_v}) \neq 0 }}
        \frac{[(K_n)_v : \Q_3]}{[K_n:\Q]}.
    \end{align}
    By \cref{lem:fldgenbycoordofnperpts}, we have $[(K_n)_v : \Q_3] = m$ for every $v \mid 3$.
    Thus, we have the inequalities 
    \begin{align}
        &\log\frac{1}{|2^n - 1|_3}
        \sum_{\substack{v \in M_{K_n} ,\  v \mid 3 \\ \Tr_{\F_{\p_v}/\F_3}(x_0 - y_0 \ \mod {\p_v}) \neq 0 }}
        \frac{[(K_n)_v : \Q_3]}{[K_n:\Q]}\\
        &= 
        \bigg(\log\frac{1}{|2^n - 1|_3} \bigg)
        \frac{ \#\big\{ v \in M_{K_n} \bigm| v \mid 3, \Tr_{\F_{\p_v}/\F_3}(x_0 - y_0 \ \mod {\p_v}) \neq 0  \big\} }{\#\big\{ v \in M_{K_n} \bigm| v \mid 3 \big\}}\\
        &>
       \frac{2}{3} \log\frac{1}{|2^n - 1|_3},
    \end{align}
    where the last inequality follows from \cref{proportionoftracezero}.
    Finally, by the Lifting-the-exponent lemma, we have
    \begin{align}
        |2^n - 1|_3 = |(-2)^n - 1^n|_3 = |(-2)-1|_3 |n|_3 = 3^{-r-1}.
    \end{align}
    Therefore, we get
    \begin{align}
        h(\f(x_0,y_0)) \geq \frac{2}{3}(r+1)\log 3,
    \end{align}
    and we are done.
\end{proof}

\begin{proof}[Proof of \cref{thm:autoA3unbddper} (\cref{unbddsetofperpt})]
    For each $r \geq 0$, let $P_{r0} \in \A^{3}(\QQ)$
    be the point $\f(x_0,y_0)$ constructed in \cref{prop:existenceperptwithhtbd}.
    Set $P_{ri} = f^i(P_{r0} )$ for $i = 0,\dots, 2\cdot 3^r - 1$.
    Ordering $\{P_{ri}\}_{(r, i)}$ by the lexicographic order of the index,
    we get a sequence $\{P_i\}_{i \geq 1}$.
    Then $\lim_{i \to \infty}h(P_i)=\infty$ 
    since we have $h(P_{ri}) \geq \frac{2}{3}(r+1)\log 3$
    by \cref{prop:existenceperptwithhtbd}.
    To end the proof, we prove that
    the set $\{P_i \mid i\geq 0\}$ is Zariski dense in $\A^3_{\QQ}$.

    Let $Q_i = \pr(P_i)$, where $\pr \colon \A^3_{\QQ} \longrightarrow \A^2_{\QQ}$
    is the projection to $(x,y)$-coordinates.
    Since $\{P_i \mid i\geq 0\}$ is $f$-invariant,
    the set $\{Q_i \mid i\geq 0\}$ is $g$-invariant.
    Moreover, the set $\{Q_i \mid i\geq 0\}$ is an infinite set since so is $\{P_i \mid i\geq 0\}$ 
    and \cref{lem:corrofperpts}.
    Recall that $g \colon \A^2_{\QQ} \longrightarrow \A^2_{\QQ}$ is an H\'enon map, and so
    it has no periodic curves (cf.\ \cite[Proposition 4.2]{BS91}).
    Thus, the set $\{Q_i \mid i\geq 0\}$ is Zariski dense in $\A^2_{\QQ}$,
    since otherwise an irreducible component of the Zariski closure of 
    $\{Q_i \mid i\geq 0\}$ must be a periodic curve.
    
    Let $Z = \overline{\{P_i \mid i\geq 0\}}$ be the Zariski closure in $\A^3_{\QQ}$
    and suppose $Z \neq \A^3_{\QQ}$.
    Then there is an irreducible component $W \subset Z$ such that
    $\pr \colon W \longrightarrow \A^2_{\QQ}$ is generically finite and dominant.
    Let $U \subset W$ be a non-empty open subset on which $\pr$ is quasi-finite.
    Then $\{P_i \mid i \geq 0\} \cap U$ is infinite and height unbounded,
    but its image $\{Q_i \mid i \geq 0, P_i \in U\}$ is height bounded (\cite[Corollary B]{Ka06}).
    By \cref{lem:htbddqfin} below, this is a contradiction. 
\end{proof}

\begin{lemma}\label{lem:htbddqfin}
    Let $\psi \colon X \longrightarrow Y$ be a quasi-finite 
    morphism between quasi-projective varieties over $\QQ$.
    For a subset $S \subset X(\QQ)$, if $\psi(S)$ is height bounded, then so is $S$.
\end{lemma}
\begin{proof}
    By Zariski Main Theorem, we may assume that $\psi$ is finite.
    Let us take open immersions $i \colon X \hookrightarrow \overline{X}$
    and $j \colon Y \hookrightarrow \overline{Y}$ so that
    $\psi$ extends to a morphism 
    $\overline{\psi} \colon \overline{X} \longrightarrow \overline{Y}$.
    Note that since $\psi$ is finite, we have $\overline{\psi}^{-1}(Y)=X$.
    Let $\L, \M$ be ample line bundles on $\overline{X}, \overline{Y}$ respectively.
    Let $k > 0$ be such that $\overline{\psi}_*\L^{-1} \otimes \M^{\otimes k}$
    is globally generated.
    Then the base locus of $\L^{-1} \otimes \psi^*\M^{\otimes k}$ is contained in 
    $\overline{X} \setminus X$.
    Indeed, since $\overline{\psi}_*(\L^{-1} \otimes \psi^*\M^{\otimes k}) = \overline{\psi}_*\L^{-1} \otimes \M^{\otimes k}$ and $\overline{\psi}^{-1}(Y) = X$,
    it reduces to the following.
    For a morphism of affine schemes $\Spec B \longrightarrow \Spec A$
    and a $B$-module $M$,
    if a subset $T \subset M$ generate $M$ as an $A$-module, then so does as a $B$-module.
    And this is obviously true.
    Therefore, we have $h_\L \leq k h_\M \circ \psi + O(1)$ on $X(\QQ)$ and we are done.
\end{proof}

\section{height boundedness of periodic points}\label{sec:htbddperpt}

In this section, we prove the height boundedness of periodic points of 
cohomologically hyperbolic maps (\cref{thm:perptishtbdd}).
We also prove an upper bound on the height of periodic points in terms of the height of the map when the map varies along a family.

Let us introduce some notation.
For a dominant rational map $f \colon X \dashrightarrow X$ on a projective variety over an algebraically closed field $k$, its cohomological Lyapunov exponents are 
\begin{align}
    \mu_i(f) \coloneqq \frac{d_i(f)}{d_{i-1}(f)}
\end{align}
for $i=1,\dots, \dim X$.
We set $\mu_{\dim X + 1}(f) = 0$.
When $k$ is not algebraically closed and $X$ is geometrically integral,
we set $\mu_i(f) = \mu_i(f_{\overline{k}})$, where $f_{\overline{k}}$
is the base change of $f$ to the algebraic closure $\overline{k}$ of $k$.

For a projective variety $X$ over an algebraically closed field,
$\widetilde{\Pic}(X)_{\R}$ denotes the colimit of $\Pic (X')_{\R}$ where $X'$ runs over birational models of $X$.
See \cite{xie2024algebraic} for the details.

\begin{proof}[Proof of \cref{thm:perptishtbdd}]
    Let us write $\mu_i = \mu_i(f)$.
    Suppose that $f$ is $p$-cohomologically hyperbolic where $1 \leq p \leq \dim X$.
    Then, we have $\mu_p > 1 > \mu_{p+1}$.

    Let $L$ be an ample divisor on $X$.
    The class in $\widetilde{\Pic}(X)_\R$ determined by $L$
    is also denoted by $L$.
    For $m \geq 0$, the symbol $(f^m)^*L$ denotes the pull-back of $L$ as an element of 
    $\widetilde{\Pic}(X)_\R$.
    By \cite[Theorem 3.7]{xie2024algebraic} (or \cite[Proposition 3.5]{MW22} when $p < \dim X$), 
    for any $\e \in (0,1)$, there is $m_\e \geq 1$ such that
    for every $m \geq m_\e$, the element
    \begin{align}
        (f^{2m})^*L + (\mu_p \mu_{p+1})^m L - (\e \mu_p)^m (f^m)^*L
    \end{align}
    is big as an element of $\widetilde{\Pic}(X)_\R$.
    We choose and fix $\e \in (0,1)$ close enough to $1$ and $m \geq m_\e$ large enough so that the following inequalities hold:
    \begin{align}
        &(\e^2\mu_p)^m + (\e^{-2}\mu_{p+1})^m \leq (\e \mu_p)^m\\
        &\e^2 \mu_p > 1 > \e^{-2}\mu_{p+1}.
    \end{align}
    Then, the element
    \begin{align}
        (f^{2m})^*L + (\e^2\mu_p)^m (\e^{-2}\mu_{p+1})^m L -  ((\e^2\mu_p)^m + (\e^{-2}\mu_{p+1})^m ) (f^m)^*L
    \end{align}
    is big.
    We set 
    \begin{align}
        \a = (\e^2\mu_p)^m,\  \b = (\e^{-2}\mu_{p+1})^m.
    \end{align}
    Consider the following commutative diagram
    \begin{equation}
        \begin{tikzcd}
            \widetilde{X} \arrow[d,"\pi",swap] \arrow[rd,"\f",swap] \arrow[rrd,"\psi"]& &\\
            X \arrow[r, "f^m",swap,dashed]& X \arrow[r, "f^m",swap,dashed] & X
        \end{tikzcd}
    \end{equation}
    where $\widetilde{X}$ is a projective variety, $\pi$ is a birational morphism that is isomorphic over $X \setminus (I_{f^m} \cup I_{f^{2m}})$, $\f, \psi$ are morphisms.
    Then, by the above argument, $\psi^*L + \a\b \pi^*L - (\a+\b)\f^*L$
    is big as an $\R$-divisor on $\widetilde{X}$.
    Let us fix a height function $h_L$ on $X(\QQ)$ associated with $L$.
    Then, there is a non-empty open subset $U \subset X \setminus (I_{f^m} \cup I_{f^{2m}})$ and $c \in \R$
    such that the inequality
    \begin{align}
        h_L \circ f^{2m} + \a \b h_L - (\a +\b) h_L \circ f^m \geq c
    \end{align}
    holds on $U(\QQ)$. Here, we use the fact that a height function associated with a big $\R$-divisor is bounded below on a non-empty Zariski open subset. This assertion follows from the fact that a big $\R$-divisor is $\R$-linearly equivalent to the sum of ample and effective $\R$-divisors. 
    Thus, we get
    \begin{align}
        &h_L \circ f^{2m} - \b h_L \circ f^m - c_1 
        \geq \a \big( h_L \circ f^{m} - \b h_L  - c_1 \big) \label{recineq1}\\
        &h_L \circ f^{2m} - \a h_L \circ f^m - c_2 
        \geq \b \big( h_L \circ f^{m} - \a h_L  - c_2 \big) \label{recineq2}
    \end{align}
    on $U(\QQ)$ where 
    \begin{align}
        c_1 = \frac{c}{1-\a}, \ 
        c_2 = \frac{c}{1-\b}.
    \end{align}
    Let us set
    \begin{align}
        &\Phi = h_L \circ f^{m} - \b h_L  - c_1 \\
        &\Psi = h_L \circ f^{m} - \a h_L  - c_2.
    \end{align}
    We claim that the set
    \begin{align}
        S = \{P \in X_f(\QQ) \mid \text{$P$ is $f$-periodic and $O_f(P) \subset U$}  \}
    \end{align}
    is bounded height.
    Let $P \in S$.
    Then, $P$ is in $X_{f^m}(\QQ)$, $P$ is $f^m$-periodic, and $O_{f^m}(P) \subset U$.
    By \cref{recineq1}, we have $\Phi(f^{mn}(P)) \geq \a^n \Phi(P)$ for all $n \geq 0$.
    As $\Phi(f^{mn}(P))$ can take only finitely many values, we get
    \begin{align}
        0 \geq \Phi(P) = h_L(f^m(P)) - \b h_L(P) - c_1. \label{ineq:htbound1}
    \end{align}
    Next, take $n \geq 1$ such that $f^{mn}(P) = P$.
    Then, by \cref{recineq2}, we obtain
    \begin{align}
        \Psi(P) = \Psi(f^{mn}(P)) \geq \b^n \Psi(P).
    \end{align}
    Thus, we have $\Psi(P) \geq 0$, i.e.
    \begin{align}
        h_L(f^m(P)) - \a h_L(P) - c_2 \geq 0.\label{ineq:htbound2}
    \end{align}
    By \cref{ineq:htbound1} and \cref{ineq:htbound2}, we get
    \begin{align}
        \a h_L(P) + c_2 \leq \b h_L(P) + c_1,
    \end{align}
    or equivalently
    \begin{align}
        h_L(P) \leq \frac{c_1 - c_2}{\a - \b}.
    \end{align}
    This inequality proves that the set $S$ is height-bounded.
\end{proof}

By \cref{thm:perptishtbdd}, there is an upper bound on the height of periodic points (whose orbit is contained in a certain Zariski open set) of a cohomologically hyperbolic map.
For a family of cohomologically hyperbolic maps $\{f_b\}_{b \in B}$, we can choose this upper bound in terms of the height of $b$.

\begin{theorem}\label{thm:easyuniformboundofhtofperpt}
    Let $\pi \colon X \longrightarrow B$ be a surjective morphism
    with a geometrically integral generic fiber between
    projective varieties over $\QQ$.
    Let $f \colon X \dashrightarrow X$ be a dominant rational map such that
    $\pi \circ f = \pi$.
    Let $L$ be a $\pi$-ample divisor on $X$, and $H$ an ample divisor on $B$.
    Let us height functions $h_L$, $h_H$ associated with them.
    Let $\eta \in B$ be the generic point,
    $X_{\eta} = \pi^{-1}(\eta)$ the generic fiber,
    and $f_{\eta} \colon X_{\eta} \dashrightarrow X_{\eta} $
    the induced dominant rational map.
    If $f_{\eta}$ is cohomologically hyperbolic, then there are non-empty
    Zariski open subsets $B^{\circ} \subset B$ and $U \subset X$
    with the following properties:
    \begin{enumerate}
        \item\label{shrinkB} For any $b \in B^{\circ}(\QQ)$, 
        $X_b\coloneqq\pi^{-1}(b)$ is irreducible and reduced,
        $X_b \not\subset I_f$, and the induced rational map $f_b \colon X_b \dashrightarrow X_b$ is dominant.
        \item There are $C_1, C_2 \in \R$ such that for every $b \in B^{\circ}(\QQ)$
        and $f_b$-periodic point $P \in (X_b)_{f_b}(\QQ)$ with $O_{f_b}(P) \subset U$, we have
        \begin{align}
            h_L(P) \leq C_1 h_H(b) + C_2.
        \end{align}
    \end{enumerate}
\end{theorem}
Here, ``$f_{\eta}$ is cohomologically hyperbolic'' means that its base change to the geometric generic fiber is cohomologically hyperbolic.

\begin{remark}
    By generic flatness, after possibly shrinking $B^{\circ}$,
    we may assume that $\pi$ is flat over $B^{\circ}$.
    Then $f|_{\pi^{-1}(B^{\circ})} \colon \pi^{-1}(B^{\circ}) \dashrightarrow \pi^{-1}(B^{\circ})$ is a family of dominant rational self-maps over $B^{\circ}$ in the sense 
    of \cite[section 5]{xie2024algebraic}.
    Then, if $f_{\eta}$ is $p$-cohomologically hyperbolic, then so are general fibers.
    Indeed, let us take real numbers $\a,\b$ so that
    \begin{align}
        d_{p-1}(f_{\eta}) < \a < d_p(f_{\eta}) > \b > d_{p+1}(f_{\eta}).
    \end{align}
    Since the map $B^{\circ} \longrightarrow \R, b \mapsto d_i(f_b)$ is continuous 
    with respect to the constructible topology (\cite[Theorem 5.1, Lemma 5.4]{xie2024algebraic}), the set
    \begin{align}
        \{b \in B^{\circ} \mid d_{p-1}(f_b) < \alpha < d_p(f_b) > \beta > d_{p+1}(f_b) \}
    \end{align}
    is open in the constructible topology.
    It contains the generic point $\eta$, so it includes a non-empty Zariski open set.
\end{remark}

\begin{remark}
    Cohomological hyperbolicity is not an open condition in the following sense.
    Consider the following family of rational maps, say over $\C$, 
    \begin{equation}
        \begin{tikzcd}[column sep=1em]
            \P^2 \times \A^1 \arrow[rr, dashed, "f"] \arrow[rd,"\pr",swap] && \P^2  \times \A^1 \arrow[ld,"\pr"] & ((x:y:1),a) \arrow[r,|->] & ((y:ax^2+y^2+x:1),a) \\
            &\A^1 & &&
        \end{tikzcd}
    \end{equation}
    where $\pr$ is the projection.
    Then $f_0 = f|_{\P^2 \times \{0\}}$ is an H\'enon map, and so it is cohomologically hyperbolic.
    But for $a \in \C \setminus \{0\}$, the map $f_a=f|_{\P^2 \times \{a\}}$ has
    dynamical degrees $d_1(f_a) = d_2(f_a)=2$.
    Here, we can see the equality $d_1(f_a)=2$ by calculating the degree sequence of $f_a$ by induction.
\end{remark}

Now, we start to prepare the proof of \cref{thm:easyuniformboundofhtofperpt}.
To avoid confusion, let us introduce some terminology and lemmas about divisors.
A Cartier divisor on a variety $X$ over a field is an element of 
$H^0(X, \mathcal{K}_X^{\times} / \O_X^{\times})$,
where $\mathcal{K}_X$ is the sheaf of meromorphic functions on $X$.
An $\R$-divisor
\footnote{This may not be standard terminology. 
Some authors refer to this as an ``$\R$-Cartier divisor''. 
As we work almost exclusively with Cartier divisors and rarely with Weil divisors,
we drop ``Cartier'' to simplify the terminology.}
on $X$ is an element of 
$H^0(X, \mathcal{K}_X^{\times} / \O_X^{\times}) \otimes_{\Z}\R$.
An $\R$-divisor $D$ is said to be effective if it can be written 
as $D = \sum_{i=1}^{r}a_i D_i$ for some $r \geq 0$,
$a_i \in \R_{\geq 0}$, and effective Cartier divisors $D_i$
(strictly speaking, $D_i$ are the images of effective Cartier divisors in $H^0(X, \mathcal{K}_X^{\times} / \O_X^{\times}) \otimes_{\Z}\R$).  
Two $\R$-divisors $D_1, D_2$ are said to be $\R$-linearly equivalent,
denoted by $D_1 \sim_\R D_2$, if
$D_1 - D_2$ is in the image of $H^0(X, \mathcal{K}_X^{\times})\otimes_{\Z}\R$.

\begin{lemma}\label{lem:bignessovernonclosedfield}
    Let $X$ be a geometrically integral projective variety
    over a field $k$.
    Let $D$ be an $\R$-divisor on $X$.
    We say that $D$ is big if its base change $D_{\overline{k}}$ is big.
    Then, the $\R$-divisor $D$ is big if and only if $D \sim_\R A+E$ for some ample $\R$-divisor $A$ and effective $\R$-divisor $E$.
\end{lemma}
\begin{proof}
    The if part is clear.
    Suppose that $D$ is big.
    We first note that $D \sim_\R D' + E'$
    for some big $\Q$-divisor $D'$ and effective $\R$-divisor $E'$.
    Indeed, write $D \sim_\R \sum_{i} a_i H_i - \sum_j b_j H'_j$
    for some $a_i,b_j \geq 0$ and effective integral divisors $H_i, H'_j$.
    Let $\e_i, \e'_j > 0$ be small real numbers so that 
    $a_i-\e_i, b_j + \e'_j \in \Q$ and $D'\coloneqq\sum_i (a_i -\e_i)H_i - \sum_j (b_j + \e'_j)H'_j$ is still big.
    Then, we have $D \sim_\R D' + \sum_i \e_i H_i  + \sum_j \e'_j H'_j$, and this proves the claim.

    Thus, we may assume that $D$ is a $\Q$-divisor.
    To this end, we may further assume that $D$ is a Cartier divisor.
    Let $A$ be an ample Cartier divisor on $X$.
    Then, since $D_{\overline{k}}$ is big, there is $n \geq 1$ such that
    $nD_{\overline{k}} - A_{\overline{k}}$ is linearly equivalent to an effective divisor. Thus
    \begin{align}
        H^0(X, \O_X(nD-A))\otimes_k \overline{k} 
        = H^0(X_{\overline{k}}, \O_X(nD_{\overline{k}}-A_{\overline{k}})) \neq 0.
    \end{align}
    This equality implies $nD - A$ is linearly equivalent to an effective divisor.
\end{proof}

\begin{lemma}\label{lem:genefftoglbeff}
    Let $\pi \colon X \longrightarrow B$ be a surjective morphism 
    between projective varieties over a field. 
    Suppose that $X$ is smooth \footnote{Normality is actually enough, but the last step of the proof becomes less trivial.}.
    Let $D$ be an $\R$-divisor on $X$ and $H$ be an ample divisor on $B$.
    If there is a non-empty open subset $U \subset B$ such that
    $D|_{\pi^{-1}(U)}$ is $\R$-linearly equivalent to an effective $\R$-divisor,
    then there is $k \geq 1$ such that $D + k\pi^*H$ is $\R$-linearly equivalent
    to an effective $\R$-divisor.
\end{lemma}
\begin{proof}
    Since $\pi^{-1}(U)$ and $X$ has the same function field, 
    there is a principal $\R$-divisor $P$ on $X$ such that 
    $(D + P)|_{\pi^{-1}(U)}$ is an effective $\R$-divisor.
    Let $[D+P]$ be the associated $\R$-Weil divisor on $X$.
    Since $[D+P]|_{\pi^{-1}(U)} = [(D + P)|_{\pi^{-1}(U)}]$ is effective,
    we can write $[D+P] = E - F$ for some effective $\R$-Weil divisors $E,F$
    such that $\Supp F \subset X \setminus \pi^{-1}(U)$.
    Thus if $k \geq 1$ is large enough, there is effective $H' \sim kH$
    such that $[\pi^*H'] - F$ is effective.
    Then $D + k\pi^*H \sim_\R D + P + \pi^*H'$ and 
    the right hand side is effective since $[D + P + \pi^*H'] = E- F + [\pi^*H']$
    is effective and $X$ is smooth. 
\end{proof}

\begin{proof}[Proof of \cref{thm:easyuniformboundofhtofperpt}]
    First note that since the generic fiber of $\pi$ is geometrically integral,
    general fiber of $\pi$ is geometrically integral.
    Take a non-empty Zariski open subset $V \subset f(X \setminus I_f)$
    and set $Z = X \setminus V$.
    Then we have $\{b \in B \mid X_b \subset I_f \cup Z\} \subset B \setminus \pi(X \setminus (I_f \cup Z))$, and the right-hand side is not Zariski dense in $B$.
    These argument implies the existence of $B^{\circ}$ satisfying (\cref{shrinkB}).

    Suppose $f_{\eta}$ is $p$-cohomologically hyperbolic.
    Write $\mu_i = \mu_i(f_{\eta})$.
    Let us denote $\kappa(\eta)$ the residue field at $\eta$, which is equal to the function field of $X$.
    Let $X_{\overline{\eta}}$, $f_{\overline{\eta}}$, and $L_{\overline{\eta}}$
    be the base change of $X,f$, and $L$ to $\overline{\kappa(\eta)}$, respectively.
    As in the proof of \cref{thm:perptishtbdd}, take $\e \in (0,1)$ and 
    $m \geq 1$ so that 
    $\a = (\e^2 \mu_p)^m$ and $\b = (\e^{-2}\mu_{p+1})^m$ satisfy $\a > 1 > \b$ and that
    \begin{align}
        &\text{
        $(f_{\overline{\eta}}^{2m})^*L_{\overline{\eta}} 
        + \a \b L_{\overline{\eta}}  
        - (\a + \b)(f_{\overline{\eta}}^m)^*L_{\overline{\eta}}$
        is big as an element of $\widetilde{\Pic}(X_{\overline{\eta}})_{\R}$. \label{recdivgeogenfib}
        }
    \end{align}

    Let us consider the following commutative diagram 
    \begin{equation}
        \begin{tikzcd}
            \widetilde{X} \arrow[d,"q"] \arrow[rd,"\f"] \arrow[rrd, "\psi",bend left=20pt]& & \\
            X \arrow[r,dashed,"f^m"] & X \arrow[r,dashed,"f^m"] & X
        \end{tikzcd}
    \end{equation} 
    where $\widetilde{X}$ is a smooth projective variety and $q$ is isomorphic over a non-empty open subset $U' \subset X \setminus (I_{f^m} \cup I_{f^{2m}})$.
    Let us set
    \begin{align}
        D = \psi^*L + \a \b q^*L - (\a+\b)\f^*L.
    \end{align}
    We note that any variety over $\kappa(\eta)$ that is birational to
    $X_{\eta}$ is geometrically integral (see, for example, \cite[3.2.2 Corollary 2.14(c)]{Liu}). In particular, $\widetilde{X}_{\eta}$ is geometrically integral.
    Therefore, the base change $D|_{\widetilde{X}_{\eta}}$ of $D$ to the geometric generic fiber of $\widetilde{X}$ is in the same class with \cref{recdivgeogenfib}.
    Hence, it is big, and by \cref{lem:bignessovernonclosedfield}, we, in particular, have
    \begin{align}
        D|_{\widetilde{X}_{\eta}} \sim_{\R}  E'
    \end{align}
    for some effective $\R$-divisor $E'$ on $\widetilde{X}_{\eta}$. 
    Then, there is a non-empty open subset $W \subset B$
    and $\R$-divisor $E$ on $(\pi \circ q)^{-1}(W)$ such that
    $E|_{\widetilde{X}_{\eta}} = E'$.
    Shrinking $W$ if necessary, we may assume that
    \begin{align}
        &D|_{(\pi \circ q)^{-1}(W)} \sim_\R  E, \text{ and}\\
        &\text{$E$ is effective.}
    \end{align} 
    By \cref{lem:genefftoglbeff}, there is $k \geq 1$ such that 
    $D + k(\pi \circ q)^*H \sim_\R F$ for some effective $\R$-divisor $F$
    on $\widetilde{X}$.

    Now, we set 
    \begin{align}
        U = U' \setminus \big(I_{f}  \cup q(\Supp F)  \big).
    \end{align}
    Then, there is $C \in \R$ such that for any $P \in U(\QQ)$, we have
    \begin{align}
        h_L \circ \psi(q^{-1}(P)) + \a \b h_L \circ q(q^{-1}(P)) - (\a + \b) h_L \circ \f(q^{-1}(P)) \geq -kh_H \circ \pi \circ q (q^{-1}(P))+ C,
    \end{align}
    or equivalently,
    \begin{align}
        h_L(f^{2m}(P)) + \a \b h_L (P) - (\a + \b) h_L(f^m(P)) \geq -kh_H(\pi(P))+ C.
    \end{align}
    Let $B^{\circ} \subset B$ be any open subset satisfying the condition 
    (\cref{shrinkB}).
    Let $b \in B^{\circ}(\QQ)$ and $P \in (X_b)_{f_b}(\QQ)$
    be such that $P$ is $f_b$-periodic and $O_{f_b}(P) \subset U$.
    Then $P \in X_f(\QQ)$, $O_f(P) \subset U \cap X_b$.
    In particular, for any point $Q \in O_{f^m}(P)$, we have
    \begin{align}
        h_L(f^{2m}(Q)) + \a \b h_L (Q) - (\a + \b) h_L(f^m(Q)) \geq -kh_H(b)+ C.
    \end{align}
    By the same calculation as in the proof of \cref{thm:perptishtbdd}, if we set
    \begin{align}
        c_1 = \frac{-kh_H(b)+ C}{1-\a},\ 
        c_2 = \frac{-kh_H(b)+ C}{1-\b},
    \end{align}
    then 
    \begin{align}
        h_L(P) \leq \frac{c_1 - c_2}{\a - \b}.
    \end{align}
    Since $\a, \b, k, C$ are determined by $X, f, L, h_L, h_H$, our assertion is proved.
\end{proof}

\section{Preperiodic points with unbounded height}
\label{sec:unbddpreperpts}

In this section, we prove \cref{thmintro:unbddpreper}.
Namely, we give an example of a cohomologically hyperbolic map
with a set of preperiodic points of unbounded height.
We first introduce a notation.
For a dominant rational map $f \colon X \dashrightarrow X$
on a projective variety $X$ over a field, let 
\begin{equation}
    \begin{tikzcd}
        \G_f \arrow[rd,"g"] \arrow[d,"\pi"]& \\
        X \arrow[r,dashed, "f",swap] & X
    \end{tikzcd}
\end{equation}
be the graph of $f$.
We say that $f$ is finite over an open subset $U \subset X$ if the map $g|_{g^{-1}(U)} \colon g^{-1}(U) \longrightarrow U$ is finite, and $g^{-1}(U) \cap \pi^{-1}(I_f) = \emptyset$.
We write
\begin{align}
    X_f^{\rm back} = \bigg\{P \in X \biggm| \parbox{20em}{for every $n \geq 1$, there is a Zariski opne neighbourhood $P \in U$ over which $f^n$ is finite.} \bigg\}.
\end{align}
See \cite[Appendix B, B.1]{matsuzawa-note-ad-dls} for the basic properties of this set.
The definition in \cite{matsuzawa-note-ad-dls} is slightly different
and gives a priori a smaller set, but the two definitions are actually equivalent. 
To see this, for $P \in X$ over which all $f^n$ are finite,
construct a sequence as in \cite[Lemma B.4]{matsuzawa-note-ad-dls} so that $P \in W_0$
by induction on $n$.

Now let $d \geq 3$, and $p$ be a prime number.
Let 
\begin{align}
    f = f_{d,p} \colon \A^2_{\QQ} \dashrightarrow \A^2_{\QQ}, (x,y) \mapsto \bigg( \frac{x^d}{y^2}-\frac{1}{p^{d-3}}, x  \bigg).
\end{align}
We use the same notation $f$ for its projectivization as well:
\begin{align}
    f  \colon \P^2_{\QQ} \dashrightarrow \P^2_{\QQ}, (x:y:z) 
    \mapsto \bigg(x^d - \frac{1}{p^{d-3}}y^2z^{d-2} : xy^2z^{d-3} : y^2 z^{d-2}  \bigg).
\end{align}

\begin{theorem}\label{thm:exampleofunbddpreper} \ 
    \begin{enumerate}
        \item\label{mainexdyndeg} We have $d_1(f) = d$, $d_2(f)=2$. In particular, the map $f$ is $1$-cohomologically hyperbolic.
        \item\label{mainexunbddpreper} Let $d = 4$, $p=2$. Let $P_0=(\frac{1+\sqrt{3}}{2}:\frac{1+\sqrt{3}}{2}:1)$.
        Then $f(P_0)=P_0$, $P_0 \in (\P^2)_f^{\rm back}$, and there is a sequence of points $\{P_n\}_{n \geq 1}$
        of $(\P^2 \setminus I_f)(\QQ)$
        such that 
        \begin{align}
            &f(P_n)= P_{n-1} \quad \text{for $n \geq 1$},\\
            &\text{$\{P_n \mid n \geq 0 \}$ is Zariski dense},\\
            &\lim_{n \to \infty} h(P_n) = \infty,
        \end{align}
        where $h$ is  the naive height function on $\P^2(\QQ)$.
    \end{enumerate}
\end{theorem}

We prove \cref{thm:exampleofunbddpreper}(\cref{mainexdyndeg}) in \cref{subsec:dyndeg}.
In \cref{subsec:backwardoffixedpt}, we prove
\cref{thm:exampleofunbddpreper}(\cref{mainexunbddpreper}).
See \cref{prop:unbddpreperrecipe} for a more general statement.
Note also that we need $d \geq 4$ in \cref{subsec:backwardoffixedpt},
but the argument in \cref{subsec:dyndeg} is valid for $d=3$ as well.

We first observe a useful property of the map $f$.
For a homogeneous polynomial $\f$ in $x,y,z$,
$V_+(\f) \subset \P^2_{\QQ}$ denotes the Zariski closed subset defined by $\f$,
and $D_+(\f) = \P^2_{\QQ} \setminus V_+(\f)$.

We use the notation 
\begin{align}
    U \coloneqq D_+((x+p^{-(d-3)}z)yz), \quad 
    V\coloneqq D_+(xyz)
\end{align}
throughout this section.

\begin{lemma}\label{lem:finitelocusoff}
    The set of indeterminacy locus of $f$ is
    \begin{align}
        I_f = \{(0:1:0), (0:0:1)\}.
    \end{align}
    Moreover, the set $U$ is the largest open subset over which $f$ is finite, and $f$ induces a finite morphism $V\longrightarrow U$.
\end{lemma}
\begin{proof}
    The first statement is clear.
    For the second statement, we first prove that $f$ induces 
    finite morphism $V \longrightarrow U$.
    The morphism $f|_{D_+(yz)} \colon D_+(yz) \longrightarrow \A^2$
    corresponds to the ring homomorphism
    \begin{align}
         \QQ[x,y] \longrightarrow \QQ[x,y]_y, x \mapsto \frac{x^d}{y^2} - \frac{1}{p^{d-3}}, y \mapsto x.
    \end{align}
    This map induces a ring homomorphism
    \begin{align}
        \f \colon \QQ[x,y]_{(x+p^{-(d-3)})y} \longrightarrow \QQ[x,y]_{xy}, x \mapsto \frac{x^d}{y^2} - \frac{1}{p^{d-3}}, y \mapsto x.
    \end{align}
    We claim that $\f$ is finite.
    To this end, it is enough to show $x,y,x^{-1}, y^{-1}$ are integral over $\Ima \f$.
    Since $x, x^{-1} \in \Ima \f$, they are integral.
    The equations
    \begin{align}
        &(y^{-1})^2 - x^{-d} \bigg( \bigg( \frac{x^d}{y^2} - \frac{1}{p^{d-3}} \bigg) + \frac{1}{p^{d-3}}  \bigg)=0\\
        &y^2 - x^{d} \bigg( \bigg( \frac{x^d}{y^2} - \frac{1}{p^{d-3}} \bigg) + \frac{1}{p^{d-3}}  \bigg)^{-1}=0
    \end{align}
    show $y, y^{-1}$ are integral.
    Thus, we have proven that $\f$ is finite, 
    and hence $f$ induces a finite morphism $V \longrightarrow U$.
    Let $W \subset \P^2$ be the largest open subset over which $f$ is finite.
    We know $U \subset W$.
    Note that we must have $W \subset f(\P^2 \setminus I_f)$.
    By the form of $f$, we can see
    \begin{align}
        f(\P^2 \setminus I_f) \cap \big(V_+(y) \cup V_+(z) \cup V_+(x+p^{-(d-3)}z) \big)
        =\{ (1:0:0), (-p^{-(d-3)} : 0 : 1)  \}.
    \end{align}
    Thus, we have $W \setminus U \subset f(\P^2 \setminus I_f)  \setminus U = \{ (1:0:0), (-p^{-(d-3)} : 0 : 1)  \}$.
    Since 
    \begin{align}
        &f|_{\P^2 \setminus I_f}^{-1}(1:0:0) \supset V_+(y) \setminus I_f\\
        &f|_{\P^2 \setminus I_f}^{-1}(-p^{-(d-3)} : 0 : 1) \supset V_+(x) \setminus I_f,
    \end{align}
    the points $(1:0:0)$ and $(-p^{-(d-3)} : 0 : 1)$ are not in $W$.
    Thus, we conclude the equality $U = W$.
\end{proof}

\subsection{Dynamical degree}\label{subsec:dyndeg}

As before, let $d \geq 3$ be an integer, $p$ a prime number, and $f= f_{d,p}$.
In this subsection, we prove the first part of \cref{thm:exampleofunbddpreper}.

\begin{proposition}\label{prop:dyndegoff}
    We have $d_1(f) = d$ and $d_2(f)=2$.
    In particular, the rational map $f$ is $1$-cohomologically hyperbolic.
\end{proposition}

\begin{proof}
    Note that for any point $P \in U(\QQ)$, $f^{-1}(P)$ consists of two $\QQ$-points.
    Thus, we have $d_2(f) = 2$.
    We show that $f$ is algebraically stable, i.e.\ $(f^n)^* = (f^*)^n$
    as endomorphisms on $\Pic \P^n_{\QQ}$ for all $n \geq 1$.
    This implies $d_1(f) = d$.
    To this end, it is enough to show that for any $n \geq 1$
    and any $f^n$-exceptional irreducible curve $C \subset \P^2_{\QQ}$,
    we have $\overline{f(C \setminus I_{f^n})} \notin I_f$.
    Here, a curve $C$ is $f^n$-exceptional means that 
    $\overline{f^n(C \setminus I_{f^n})}$ is a point. 
    
    By \cref{lem:finitelocusoff}, $f$ contracts at most three curves 
    $D=V_+(x), V_+(y), V_+(z)$ and their images $\overline{f(D \setminus I_f)}$ are
    \begin{align}
        {
        \renewcommand{\arraystretch}{2}
        \begin{array}{rcl}
            V_+(x) & \xlongrightarrow{\quad f \quad} & \{(-p^{-(d-3)}:0:1)\},\\
            V_+(y) & \xlongrightarrow{\quad f \quad} & \{(1:0:0)\},\\
            V_+(z) & \xlongrightarrow{\quad f \quad} &
            \begin{cases}
            \{(1:0:0)\} & \text{if } d\ge 4,\\
            V_+(z) & \text{if } d=3.
            \end{cases}
        \end{array}
        }
    \end{align}
    Moreover, we have 
    \begin{align}
        f(-p^{-(d-3)} : 0 : 1) = (1:0:0),\  f(1:0:0)=(1:0:0). \label{eq:imagesofcontractedpointsbyf}
    \end{align}
    Therefore, if $C \subset \P^2_{\QQ}$ is an $f^n$-exceptional curve,
    then there is some $0 \leq k \leq n$ such that $f^k(C \setminus I_{f^k})$
    is either $(1:0:0)$ or $(-p^{-(d-3)} : 0 : 1)$.
    By \cref{eq:imagesofcontractedpointsbyf}, $\overline{f(C \setminus I_{f^n})}$
    is either $(1:0:0)$ or $(-p^{-(d-3)} : 0 : 1)$.
    As these points are not contained in $I_f=\{(0:1:0), (0:0:1)\}$, the assertion is proved.
\end{proof}

\subsection{Backward orbit of fixed points}\label{subsec:backwardoffixedpt}

In this section, we assume $d \geq 4$.

\begin{lemma}\label{lem:backwardorbitpadicabs}
    Let $x_0 \in \QQ$ be such that $x_0^{d-2} - x_0 - p^{-(d-3)}=0$.
    Then, the point $P_0\coloneqq(x_0:x_0:1) \in \P^2(\QQ)$ is not in $I_f$, and $f(P_0)=P_0$ holds.
    For $n \geq 1$, we inductively choose $P_n \in (\P^2 \setminus I_f)(\QQ)$ 
    so that $f(P_n) = P_{n-1}$ if exists. 
    For $n \geq 0$, we have 
    \begin{itemize}
        \item[$(1)_n$] $P_n$ exists and $P_n \in \A^2(\QQ)$;
        \item[$(2)_n$] If we write $P_n = (x_n:y_n:1)$, then $|x_n|_p=|y_n|_p=p^{(d-3)/(d-2)}$,
        where $|\ |_p$ is any extension of the $p$-adic absolute value of $\Q$ to $\QQ$. 
        (Recall that we normalize $|\ |_p$ so that $|p|_p=p^{-1}$.)
    \end{itemize}
    In particular, $P_n \in U$ for all $n \geq 0$.
\end{lemma}
\begin{proof}
    It is clear that $P_0 \notin I_f$.
    We also have
    \begin{align}
        f(P_0) = \bigg(  \frac{x_0^d}{x_0^2} - \frac{1}{p^{d-3}} : x_0 :1 \bigg)
        = (x_0:x_0:1).
    \end{align}
    We prove $(1)_n, (2)_n$ by induction on $n$.
    $(1)_0$ is clear.
    To see $(2)_0$, let us set $x_0' = px_0$.
    Then ${x_0'}^{d-2} - p^{d-3}x_0' - p=0$.
    Let $R$ be the ring of integers of $\Q(x_0')$.
    Then $x_0' \in pR$. Hence $|{x_0'}^{d-2}|_p = |p^{d-3}x_0' + p|_p = |p|_p.$
    Thus $|x_0^{d-2}|_p = |p|_p^{-(d-3)} = p^{d-3}$.
    Thus $|x_0|_p = p^{(d-3)/(d-2)}$.
    Since $y_0=x_0$, this proves $(2)_0$.

    Let $n \geq 0$ and suppose that $(1)_n$ and $(2)_n$ hold.
    By $(2)_n$, we have $P_n \in U$.
    Thus $P_{n+1}$ exists and contained in $V(\QQ)$.
    In particular, the point $P_{n+1}$ is in $\A^2(\QQ)$.
    Since 
    \begin{align}
        \frac{x_{n+1}^d}{y_{n+1}^2} - \frac{1}{p^{d-3}} = x_n, \ x_{n+1} = y_n,
    \end{align}
    we obviously have $|x_{n+1}|_p = |y_n|_p = p^{(d-3)/(d-2)}$.
    Also, we have
    \begin{align}
        |y_{n+1}|_p^2 = \frac{|x_{n+1}|_p^d}{|x_n + p^{-(d-3)}|_p}
        = \frac{|y_{n}|_p^d}{|x_n + p^{-(d-3)}|_p}
        = p^{d(d-3)/(d-2)} p^{-(d-3)} = p^{2(d-3)/(d-2)},
    \end{align}
    and thus $|y_{n+1}|_p = p^{(d-3)/(d-2)}$.
\end{proof}

\begin{corollary}\label{cor: sufficient condition for f-back}
    For $x_0 \in \QQ$ with  $x_0^{d-2} - x_0 - p^{-(d-3)}=0$,
    we have $(x_0:x_0:1) \in (\P^2)_f^{\rm back}$.
\end{corollary}
\begin{proof}
    By \cref{lem:backwardorbitpadicabs}, $f^{-n}(x_0:x_0:1) \subset U$
    for all $n \geq 0$. This implies that $(x_0:x_0:1)$ is in $(\P^2)_f^{\rm back}$.
\end{proof}

\begin{lemma}\label{lem:preimageabsogrowth}
    Let $(a,b), (a',b') \in \A^2(\QQ)$ be such that
    \begin{align}
        &b, b'\neq 0\\
        &f(a',b') = \bigg( \frac{{a'}^d}{{b'}^2} - \frac{1}{p^{d-3}}, a'\bigg) = (a,b).
    \end{align}
    Let $|\ |$ be an arbitrary absolute value of $\QQ$ extending the archimedean absolute value of $\Q$.
    Let $t \in \R$ and suppose either 
    \begin{align}
        \text{$d\geq 5$ and $t \geq 1$, or
        $d=4$ and $t \geq \frac{1+\sqrt{3}}{2}$}.
    \end{align}
    If
    \begin{align}
        \frac{1}{1+\frac{1}{2t}}|b| \geq  |a| \geq t,
    \end{align}
    then
    \begin{align}
        \frac{1}{1+\frac{1}{2t}}|b'| \geq  |a'| \geq  t,\ \text{and}\quad  |b'| \geq  |b|^{(d-1)/2}.
    \end{align}
\end{lemma}
\begin{proof}
    First note that $|a'| = |b| \geq  (1+\frac{1}{2t})t > t$.
    Next, we have
    \begin{align}
        |b'|^2 = \bigg| \frac{b^d}{a + p^{-(d-3)}} \bigg| \geq \frac{|b|^d}{|a| + 1/2}
        \geq  \frac{|b|^d}{(1 + \frac{1}{2t})|a|} \geq |b|^{d-1}.
    \end{align}
    Thus, we get
    \begin{align}\label{eq: inductionstep}
        |b'| \geq  |b|^{(d-1)/2} = |a'|^{(d-1)/2} = |a'| |a'|^{(d-3)/2} 
        \geq |a'| \big((1+\tfrac{1}{2t})t\big)^{(d-3)/2}.
    \end{align}
    If $d \geq 5$ and $t \geq 1$, we have $((1+\frac{1}{2t})t)^{(d-3)/2} \geq 1+\frac{1}{2t}$.
    If $d = 4$ (and $t > 0$), the inequality $((1+\frac{1}{2t})t)^{(d-3)/2} \geq 1+\frac{1}{2t}$ is equivalent to
    $t \geq 1 + \frac{1}{2t}$.
    This holds if $t \geq (1 + \sqrt{3})/2$.
    Hence, the inequality \cref{eq: inductionstep} implies our assertion.
\end{proof}

\begin{proposition}\label{prop:unbddpreperrecipe}
    Let $x_0 \in \QQ$ be such that $x_0^{d-2} - x_0 - p^{-(d-3)}=0$.
    Then by \cref{lem:backwardorbitpadicabs}, there is a sequence of points
    $P_n \in U(\QQ)$ $(n \geq 0)$ such that
    \begin{align}
        P_0=(x_0,x_0), \quad
        f(P_n) = P_{n-1} \quad \text{for $n \geq 1$}.
    \end{align}
    Let us write $P_n=(x_n,y_n)$.
    Let $t$ be a real number and suppose either 
    \begin{align}
        \text{$d\geq 5$ and $t \geq 1$, or
        $d=4$ and $t \geq \frac{1+\sqrt{3}}{2}$}.
    \end{align}
    If there is $n \geq 0$ and an archimedean absolute value $|\ |$ on $\QQ$ extending the archimedean absolute value of $\Q$ such that
    \begin{align}
        \frac{1}{1+\frac{1}{2t}}|y_n| \geq |x_n| \geq t, \label{ineq:initialcond}
    \end{align}
    then we have
    \begin{align}
        \liminf_{n \to \infty}h(P_n)^{1/n} \geq \frac{d-1}{2},
    \end{align}
    and $\{P_n \mid n \geq 0\}$ is Zariski dense in $\P^2$.
    Here, recall that $h$ denotes the naive height function on $\P^2(\QQ)$.
\end{proposition}
\begin{proof}
    Let $K_n = \Q(x_n,y_n)$ be the field generated by $x_n,y_n$ over $\Q$.
    Since $f(x_{n+1}:y_{n+1}:1)=(x_n:y_n:1)$, we have a tower of field extensions $K_0 \subset K_1 \subset K_2 \subset \cdots$.
    Let us write $|\ |_v = |\ ||_{K_n}$, the restriction of $|\ |$ on $K_n$.
    Then, by \cref{lem:preimageabsogrowth}, for any $m > n$ and any absolute value
    $|\ |_w$ on $K_m$ extending $|\ |_v$, we have
    \begin{align}
        |y_m|_w \geq |y_n|_v^{((d-1)/2)^{m-n}}.
    \end{align}
    Thus, we get the following inequalities
    \begin{align}
        h(P_m) 
        &\geq \sum_{\substack{w \in M_{K_m} \\ w \mid v}} 
        \frac{[(K_m)_w : \R]}{[K_m: \Q]} \log|y_m|_w\\
        &\geq \sum_{\substack{w \in M_{K_m} \\ w \mid v}} 
        \frac{[(K_m)_w : (K_n)_v][(K_n)_v:\R]}{[K_m: K_n][K_n:\Q]}
        \bigg(\frac{d-1}{2} \bigg)^{m-n}\log |y_n| \\
        &= \frac{[(K_n)_v:\R]}{[K_n:\Q]}\bigg(\frac{d-1}{2} \bigg)^{m-n}\log |y_n|\\
        & \geq \frac{[(K_n)_v:\R]}{[K_n:\Q]}\bigg(\frac{d-1}{2} \bigg)^{m-n}\log(1+1/2t)t.
    \end{align}
    Hence, we obtain the inequality
    \begin{align}
        \liminf_{n \to \infty}h(P_n)^{1/n} \geq \frac{d-1}{2} (> 1). \label{eq:lowerboundofhtalongback}
    \end{align}

    Next, we prove that the set $S = \{P_n \mid n \geq 0\}$ is Zariski dense in $\P^2$.
    Let $Z \subset \P^2$ be the Zariski closure of $S$.
    Suppose $Z \subsetneq \P^2$.
    Since $S$ is an infinite set, we have $\dim Z=1$.
    Note that $S \subset \P^2_f(\QQ)$ and $f(S)= S$, and hence
    $f(Z \setminus I_f)$ is a dense subset of $Z$.
    Let $Z_0 \subset Z$ be a $1$-dimensional irreducible component.
    Since $Z_0 \cap \P^2_f(\QQ) \cap S$ is infinite, 
    the closed subset $\overline{f^i(Z_0 \cap \P^2_f(\QQ) )}$ is $1$-dimensional irreducible component of $Z$
    for all $i \geq 0$. Moreover, any point in $S$ is contained in $f^i(Z_0 \cap \P^2_f(\QQ))$ for some $i$.
    Thus, every irreducible component of $Z$ is 
    of the form $\overline{f^i(Z_0 \cap \P^2_f(\QQ) )}$. Since $f(Z \setminus I_f)$
    is dense in $Z$, generic points of $Z$ form a periodic cycle, say period $r$, under $f$.
    In particular, the closed subset $Z$ is pure dimension one.
    Consider the infinite set $\{P_{nr} \mid n \geq 0\}$.
    Let $C \subset Z$ be an irreducible component containing infinitely many $P_{nr}$.
    Since $f^r(C \cap \P^2_f(\QQ)) \subset C$, we have $P_{nr} \in C$ for all $n \geq 0$.
    Consider the following diagram
    \begin{equation}
        \begin{tikzcd}
            \widetilde{C} \arrow[r, "g"] \arrow[d, "\nu",swap] & \widetilde{C} \arrow[d, "\nu"]\\
            C \arrow[r,dashed , "f^r"]& C\\
            P_{(n+1)r} \arrow[r,|->] \arrow[u,phantom,"\in", sloped] & P_{nr} \arrow[u,phantom,"\in", sloped] 
        \end{tikzcd}
    \end{equation}
    where $\nu$ is the normalization of $C$ and $g$ is the induced morphism.
    Since $g(\nu^{-1}(P_{(n+1)r})) \subset \nu^{-1}(P_{nr})$,
    \begin{align}
        \nu^{-1}(\{P_{nr} \mid n \geq 0\}) \subset \bigcup_{n \geq 0}g^{-n}(\nu^{-1}(P_0)).
    \end{align}
    As $g(\nu^{-1}(P_0)) \subset \nu^{-1}(P_0)$, the right-hand side is the iterated preimages of finitely many preperiodic points of $g$.
    Since this set is infinite, the morphism $g$ is not an automorphism.
    As $\widetilde{C}$ is an irreducible curve, the morphism $g$ is a polarized endomorphism, and hence the set of preperiodic points in $\widetilde{C}(\QQ)$ is a set of bounded height.
    This implies that $h \circ \nu$ is bounded on $\nu^{-1}(\{P_{nr} \mid n \geq 0\})$,
    but this contradicts \cref{eq:lowerboundofhtalongback}.

\end{proof}

\begin{proof}[Proof of \cref{thm:exampleofunbddpreper}(\cref{mainexunbddpreper})]
    We check the assumptions of \cref{prop:unbddpreperrecipe} for 
    a certain choice of $x_0$ and $n$ when $d = 4$ and $p = 2$.
    We fix an embedding $\QQ \subset \C$ and let $|\ |$ be the usual absolute value of $\C$.
    The equation in \cref{cor: sufficient condition for f-back} for our situation is $x^2 - x - 1/2 = 0$.
    We choose the larger root $x_0 = (1 + \sqrt{3})/2$.
    We can check that
    \begin{equation}
        \begin{tikzcd}
            (x_0,x_0) & \arrow[l,"f", maps to] (x_0,-x_0) & \arrow[l,"f", maps to] (-x_0,x_0) 
            & \arrow[l,"f", maps to] \bigg(x_0, \left( \frac{x_0^4}{1/2 - x_0} \right)^{1/2} \bigg) =(x_3, y_3)
        \end{tikzcd}
    \end{equation}
    where we choose one of the square roots at the right-most point.
    The inequality \cref{ineq:initialcond} for $|\ |$, $t = (1+\sqrt{3})/2 (=x_0)$, and $(x_3,y_3)$ is equivalent to $|x_3| \geq t$ and
    \begin{align}
        \frac{x_0^4}{x_0 - 1/2} \geq (1 + 1/2t)^2x_0^2,
    \end{align}
    or
    \begin{align}
        \sqrt{\frac{2}{\sqrt{3}}} x_0^2 \geq x_0 + \frac{1}{2} (= x_0^2).
    \end{align}
    Since this is true, our assertion follows from \cref{prop:unbddpreperrecipe}.
\end{proof}

We remark that the backward orbit of $(x_0,x_0)$ constructed in the
proof of \cref{thm:exampleofunbddpreper}(\cref{mainexunbddpreper})
cannot be defined over number fields of bounded degree.
More precisely, let $d = 4, p=2$ so that
\begin{align}
    f \colon \A^2_{\QQ} \dashrightarrow \A^2_{\QQ}, 
    (x,y) \mapsto \bigg(\frac{x^4}{y^2} - \frac{1}{2}, x \bigg).
\end{align}
Let $x_0 = (1+\sqrt{3})/2$ and $P_0 = (x_0,x_0)$.
As in the proof of \cref{thm:exampleofunbddpreper}(\cref{mainexunbddpreper}), we choose
\[
    P_3 = (x_3, y_3) \coloneqq \left(x_0, \left( \frac{x_0^4}{1/2 - x_0} \right)^{1/2} \right)
\]
so that $f^3(P_3) = P_0$,
and pick any $P_n = (x_n,y_n)$ such that $f(P_n) = P_{n-1}$ for $n \geq 4$.
Then, the following proposition holds.
\begin{proposition}\label{prop:backorbitdeffldunbdd}
    We have $[\Q(x_n, y_n) : \Q] \to \infty$ as $n \to \infty$.
\end{proposition}
\begin{proof}
    Let us write $K_n = \Q(x_n, y_n)$.
    Note that we have $K_0 \subset K_1 \subset \cdots$.
    We prove that the value groups at primes over $3$ extend infinitely many times.
    Let us fix an absolute value $|\ |_3$ on $\QQ$ that is an extension
    of the $3$-adic absolute value of $\Q$.
    Then set $v(\cdot) = - \log |\cdot|_3$.
    This is a valuation on $\QQ$ such that $v(3)=1$.

    We first observe that
    \begin{align}
        v(x_4) = v(y_3) = -\frac{1}{4},\  v(y_4) = 2v(y_3) = -\frac{1}{2}. \label{valofP4}
    \end{align}
    For $n \geq 4$, recalling the formula
    \begin{align}
        (x_{n+1}, y_{n+1}) = \bigg( y_n, \frac{y_n^2}{(x_n + 1/2)^{1/2}} \bigg),
    \end{align}
    we have
    \begin{align}
        v(x_{n+1}) = v(y_n),\ v(y_{n+1}) = 2v(y_n) - \frac{1}{2} v(x_n + 1/2).
    \end{align}
    As $v(1/2)=0$, if $v(x_n) < 0$, we have 
    \begin{align}
        v(y_{n+1}) = 2v(y_n) - \frac{1}{2} v(x_n)
    \end{align}
    so that
    \begin{align}
        \begin{pmatrix} v(x_{n+1}) \\ v(y_{n+1})\end{pmatrix}
        =
        \begin{pmatrix}
            0 & 1 \\
            -1/2 & 2
        \end{pmatrix}
        \begin{pmatrix}
            v(x_n) \\ v(y_n)
        \end{pmatrix}.
    \end{align}
    To repeat this argument, we need $v(x_{n+1}) < 0$.
    To this end, we note that $v(x_n) < 0$ and $v(y_n) \leq (1-1/\sqrt{2})v(x_n)$
    implies $v(x_{n+1}) < 0$ and $v(y_{n+1}) \leq (1-1/\sqrt{2})v(x_{n+1})$.
    Indeed, $v(x_{n+1}) = v(y_n) \leq (1-1/\sqrt{2})v(x_n) < 0$.
    Also,
    \begin{align}
        v(y_{n+1}) = 2 v(y_n) - \frac{1}{2} v(x_n) \leq 
        \bigg( 2 - \frac{1}{2 - \sqrt{2}}  \bigg) v(y_n)
        = (1 - 1 / \sqrt{2}) v(x_{n+1}).
    \end{align}
    (To get the number $1 - 1/\sqrt{2}$, diagonalize the coefficient matrix.)
    By \cref{valofP4}, the inequalities $v(x_n) < 0$ and $v(y_n) \leq (1-1/\sqrt{2})v(x_n)$ are satisfied when $n=4$.
    Thus, we get
    \begin{align}
        \begin{pmatrix} v(x_{n}) \\ v(y_{n})\end{pmatrix}
        =
        \begin{pmatrix}
            0 & 1 \\
            -1/2 & 2
        \end{pmatrix}^{n-4}
        \begin{pmatrix}
            v(x_4) \\ v(y_4)
        \end{pmatrix}
        =
        \begin{pmatrix}
            0 & 1 \\
            -1/2 & 2
        \end{pmatrix}^{n-4}
        \begin{pmatrix}
            -1/4 \\ -1/2
        \end{pmatrix}
        \label{valofxnyn}
    \end{align}
    for $n \geq 4$.
    Then 
    \begin{align}
        \begin{pmatrix}
            0 & 1 \\
            -1/2 & 2
        \end{pmatrix}^2
        =
        \begin{pmatrix}
            -1/2 & 2 \\
            -1 & 7/2
        \end{pmatrix}
    \end{align}
    and if we write 
    \begin{align}
        \begin{pmatrix}
            0 & 1 \\
            -1/2 & 2
        \end{pmatrix}^{2m}
        =
        \begin{pmatrix}
            a_m & b_m \\
            c_m & d_m
        \end{pmatrix}
    \end{align}
    for $m \geq 1$,
    then the entries $a_m, b_m, c_m, d_m$ lie in $\Z[1/2]$, 
    and induction on $m$ shows
    \begin{align}
        v_2(a_m) = v_2(d_m) = -m,\ v_2(b_m), v_2(c_m) > -m
    \end{align}
    where $v_2$ is the $2$-adic valuation.
    By \cref{valofxnyn}, we get
    \begin{align}
        v_2(v(x_{2m+4})) = -m - 2
    \end{align}
    for all $m \geq 1$.
    This equality implies the inclusion $v(K_{2m+4}^{\times}) \supset \frac{1}{2^{m+2}} \Z$.
    Hence, our assertion follows from the fundamental equality of the extension degree.
\end{proof}

\bibliographystyle{acm}
\bibliography{periodic_coh_hyp}

@article {Mar03,
    AUTHOR = {Marcello, Sandra},
     TITLE = {Sur la dynamique arithm\'etique des automorphismes de l'espace
              affine},
   JOURNAL = {Bull. Soc. Math. France},
  FJOURNAL = {Bulletin de la Soci\'et\'e{} Math\'ematique de France},
    VOLUME = {131},
      YEAR = {2003},
    NUMBER = {2},
     PAGES = {229--257},
      ISSN = {0037-9484,2102-622X},
   MRCLASS = {11G50 (32M17 37F10)},
  MRNUMBER = {1988948},
MRREVIEWER = {Yuri\ Tschinkel},
       DOI = {10.24033/bsmf.2441},
       URL = {https://doi.org/10.24033/bsmf.2441},
}

@article {Mar00,
    AUTHOR = {Marcello, Sandra},
     TITLE = {Sur les propri\'et\'es arithm\'etiques des it\'er\'es
              d'automorphismes r\'eguliers},
   JOURNAL = {C. R. Acad. Sci. Paris S\'er. I Math.},
  FJOURNAL = {Comptes Rendus de l'Acad\'emie des Sciences. S\'erie I.
              Math\'ematique},
    VOLUME = {331},
      YEAR = {2000},
    NUMBER = {1},
     PAGES = {11--16},
      ISSN = {0764-4442},
   MRCLASS = {11G50},
  MRNUMBER = {1780177},
MRREVIEWER = {Laurent\ Denis},
       DOI = {10.1016/S0764-4442(00)00325-6},
       URL = {https://doi.org/10.1016/S0764-4442(00)00325-6},
}

@article {Den95,
    AUTHOR = {Denis, Laurent},
     TITLE = {Points p\'eriodiques des automorphismes affines},
   JOURNAL = {J. Reine Angew. Math.},
  FJOURNAL = {Journal f\"ur die Reine und Angewandte Mathematik. [Crelle's
              Journal]},
    VOLUME = {467},
      YEAR = {1995},
     PAGES = {157--167},
      ISSN = {0075-4102,1435-5345},
   MRCLASS = {14E05 (11G99)},
  MRNUMBER = {1355926},
MRREVIEWER = {Franz\ Halter-Koch},
       DOI = {10.1515/crll.1995.467.157},
       URL = {https://doi.org/10.1515/crll.1995.467.157},
}

@article {Sil94,
    AUTHOR = {Silverman, Joseph H.},
     TITLE = {Geometric and arithmetic properties of the {H}\'enon map},
   JOURNAL = {Math. Z.},
  FJOURNAL = {Mathematische Zeitschrift},
    VOLUME = {215},
      YEAR = {1994},
    NUMBER = {2},
     PAGES = {237--250},
      ISSN = {0025-5874,1432-1823},
   MRCLASS = {14G25 (11G99 14E05 58F20)},
  MRNUMBER = {1259460},
MRREVIEWER = {Takeshi\ Ooe},
       DOI = {10.1007/BF02571713},
       URL = {https://doi.org/10.1007/BF02571713},
}

@article {BS91,
    AUTHOR = {Bedford, Eric and Smillie, John},
     TITLE = {Polynomial diffeomorphisms of {${\bf C}^2$}: currents,
              equilibrium measure and hyperbolicity},
   JOURNAL = {Invent. Math.},
  FJOURNAL = {Inventiones Mathematicae},
    VOLUME = {103},
      YEAR = {1991},
    NUMBER = {1},
     PAGES = {69--99},
      ISSN = {0020-9910,1432-1297},
   MRCLASS = {32H50 (32C30 32L30 58F11 58F15 58F23)},
  MRNUMBER = {1079840},
MRREVIEWER = {Feliks\ Przytycki},
       DOI = {10.1007/BF01239509},
       URL = {https://doi.org/10.1007/BF01239509},
}

@book {Fuetale,
    AUTHOR = {Fu, Lei},
     TITLE = {Etale cohomology theory},
    SERIES = {Nankai Tracts in Mathematics},
    VOLUME = {14},
   EDITION = {Revised},
 PUBLISHER = {World Scientific Publishing Co. Pte. Ltd., Hackensack, NJ},
      YEAR = {2015},
     PAGES = {x+611},
      ISBN = {978-981-4675-08-6},
   MRCLASS = {14F20 (11G25)},
  MRNUMBER = {3380806},
       DOI = {10.1142/9569},
       URL = {https://doi.org/10.1142/9569},
}

@article {MX25,
    AUTHOR = {Matsuzawa, Yohsuke and Xie, Junyi},
     TITLE = {Arithmetic degree and its application to {Z}ariski dense orbit
              conjecture},
   JOURNAL = {J. Lond. Math. Soc. (2)},
  FJOURNAL = {Journal of the London Mathematical Society. Second Series},
    VOLUME = {112},
      YEAR = {2025},
    NUMBER = {3},
     PAGES = {Paper No. e70282, 17},
      ISSN = {0024-6107,1469-7750},
   MRCLASS = {37P15 (37P55)},
  MRNUMBER = {4958582},
       DOI = {10.1112/jlms.70282},
       URL = {https://doi.org/10.1112/jlms.70282},
}

@book {silADS,
    AUTHOR = {Silverman, Joseph H.},
     TITLE = {The arithmetic of dynamical systems},
    SERIES = {Graduate Texts in Mathematics},
    VOLUME = {241},
 PUBLISHER = {Springer, New York},
      YEAR = {2007},
     PAGES = {x+511},
      ISBN = {978-0-387-69903-5},
   MRCLASS = {11-02 (11-01 11G05 11G07 11G50 37-02 37F10)},
  MRNUMBER = {2316407},
MRREVIEWER = {Thomas\ Ward},
       DOI = {10.1007/978-0-387-69904-2},
       URL = {https://doi.org/10.1007/978-0-387-69904-2},
}

@book {Liu,
    AUTHOR = {Liu, Qing},
     TITLE = {Algebraic geometry and arithmetic curves},
    SERIES = {Oxford Graduate Texts in Mathematics},
    VOLUME = {6},
      NOTE = {Translated from the French by Reinie Ern\'e,
              Oxford Science Publications},
 PUBLISHER = {Oxford University Press, Oxford},
      YEAR = {2002},
     PAGES = {xvi+576},
      ISBN = {0-19-850284-2},
   MRCLASS = {14-01 (11G30 14A05 14A15 14Gxx 14Hxx)},
  MRNUMBER = {1917232},
MRREVIEWER = {C\'icero\ Carvalho},
}

@article {matsuzawa-note-ad-dls,
    AUTHOR = {Matsuzawa, Yohsuke},
     TITLE = {Existence of arithmetic degrees for generic orbits and
              dynamical {L}ang-{S}iegel problem},
   JOURNAL = {J. Reine Angew. Math.},
  FJOURNAL = {Journal f\"ur die Reine und Angewandte Mathematik. [Crelle's
              Journal]},
    VOLUME = {825},
      YEAR = {2025},
     PAGES = {305--335},
      ISSN = {0075-4102,1435-5345},
   MRCLASS = {37P30 (11G50 37F80)},
  MRNUMBER = {4939945},
       DOI = {10.1515/crelle-2025-0038},
       URL = {https://doi.org/10.1515/crelle-2025-0038},
}

@article{xie2024algebraic,
  title={Algebraic dynamics and recursive inequalities},
  author={Xie, Junyi},
  journal={arXiv preprint arXiv:2402.12678},
  year={2024}
}

@article{DS05,
	author = {Dinh, Tien-Cuong and Sibony, Nessim},
	doi = {10.4007/annals.2005.161.1637},
	fjournal = {Annals of Mathematics. Second Series},
	issn = {0003-486X},
	journal = {Ann. of Math. (2)},
	mrclass = {32H50 (37B40 37F05)},
	mrnumber = {2180409},
	mrreviewer = {Jean-Yves Briend},
	number = {3},
	pages = {1637--1644},
	title = {Une borne sup\'{e}rieure pour l'entropie topologique d'une application rationnelle},
	url = {https://doi.org/10.4007/annals.2005.161.1637},
	volume = {161},
	year = {2005}}

@article{Wa22,
	author = {Long, WANG},
	journal = {Journal of the Mathematical Society of Japan},
	number = {1},
	pages = {1--26},
	publisher = {Mathematical Society of Japan},
	title = {Periodic points and arithmetic degrees of certain rational self-maps},
	volume = {1},
	year = {2023}}

@article {MW22,
    AUTHOR = {Matsuzawa, Yohsuke and Wang, Long},
     TITLE = {Arithmetic degrees and {Z}ariski dense orbits of
              cohomologically hyperbolic maps},
   JOURNAL = {Trans. Amer. Math. Soc.},
  FJOURNAL = {Transactions of the American Mathematical Society},
    VOLUME = {377},
      YEAR = {2024},
    NUMBER = {9},
     PAGES = {6311--6340},
      ISSN = {0002-9947,1088-6850},
   MRCLASS = {37P15 (37P05 37P30 37P55)},
  MRNUMBER = {4855313},
MRREVIEWER = {Ruofan\ Li},
       DOI = {10.1090/tran/9211},
       URL = {https://doi.org/10.1090/tran/9211},
}

@article{Ka06,
	author = {Kawaguchi, Shu},
	doi = {10.1007/s00208-006-0750-y},
	fjournal = {Mathematische Annalen},
	issn = {0025-5831},
	journal = {Math. Ann.},
	mrclass = {11G50 (14G40 32H50 37F10)},
	mrnumber = {2221115},
	mrreviewer = {Joseph H. Silverman},
	number = {2},
	pages = {285--310},
	title = {Canonical height functions for affine plane automorphisms},
	url = {https://doi.org/10.1007/s00208-006-0750-y},
	volume = {335},
	year = {2006}}

@article{Tr20,
	author = {Truong, Tuyen Trung},
	doi = {10.1515/crelle-2017-0052},
	fjournal = {Journal f\"{u}r die Reine und Angewandte Mathematik. [Crelle's Journal]},
	issn = {0075-4102},
	journal = {J. Reine Angew. Math.},
	mrclass = {37P05 (14E05 37F05 37F80)},
	mrnumber = {4048444},
	mrreviewer = {Yu Yasufuku},
	pages = {139--182},
	title = {Relative dynamical degrees of correspondences over a field of arbitrary characteristic},
	url = {https://doi.org/10.1515/crelle-2017-0052},
	volume = {758},
	year = {2020}}

@article{Da20,
	author = {Dang, Nguyen-Bac},
	doi = {10.1112/plms.12366},
	fjournal = {Proceedings of the London Mathematical Society. Third Series},
	issn = {0024-6115},
	journal = {Proc. Lond. Math. Soc. (3)},
	mrclass = {14E05 (37C35 37P55)},
	mrnumber = {4133708},
	mrreviewer = {Alexandr V. Pukhlikov},
	number = {5},
	pages = {1268--1310},
	title = {Degrees of iterates of rational maps on normal projective varieties},
	url = {https://doi.org/10.1112/plms.12366},
	volume = {121},
	year = {2020}}

@book{BG06,
	author = {Bombieri, Enrico and Gubler, Walter},
	doi = {10.1017/CBO9780511542879},
	isbn = {978-0-521-84615-8; 0-521-84615-3},
	mrclass = {11G50 (11-02 11G10 11G30 11J68 14G40)},
	mrnumber = {2216774},
	mrreviewer = {Yuri Bilu},
	pages = {xvi+652},
	publisher = {Cambridge University Press, Cambridge},
	series = {New Mathematical Monographs},
	title = {Heights in {D}iophantine geometry},
	url = {https://doi.org/10.1017/CBO9780511542879},
	volume = {4},
	year = {2006}}

@book{La83,
	author = {Lang, Serge},
	doi = {10.1007/978-1-4757-1810-2},
	isbn = {0-387-90837-4},
	mrclass = {11-02 (11Dxx 11Gxx 14G25)},
	mrnumber = {715605},
	mrreviewer = {Gerd Faltings},
	pages = {xviii+370},
	publisher = {Springer-Verlag, New York},
	title = {Fundamentals of {D}iophantine geometry},
	url = {https://doi.org/10.1007/978-1-4757-1810-2},
	year = {1983}}

@book{HS00,
	author = {Hindry, Marc and Silverman, Joseph H.},
	doi = {10.1007/978-1-4612-1210-2},
	isbn = {0-387-98975-7; 0-387-98981-1},
	mrclass = {11Gxx (11-02 11G10 11G30 11G50 14G25)},
	mrnumber = {1745599},
	mrreviewer = {Dino J. Lorenzini},
	note = {An introduction},
	pages = {xiv+558},
	publisher = {Springer-Verlag, New York},
	series = {Graduate Texts in Mathematics},
	title = {Diophantine geometry},
	url = {https://doi.org/10.1007/978-1-4612-1210-2},
	volume = {201},
	year = {2000}}

\end{document}